\documentclass[a4paper,11pt]{amsart}

\usepackage{amssymb,latexsym}
\usepackage{amsthm}
\usepackage{amsmath}
\usepackage{amsfonts}
\usepackage{mathrsfs}
\usepackage{mathdots}
\usepackage{graphicx}
\usepackage[all]{xy}
\usepackage{chngcntr}
\usepackage{fancyhdr}
\usepackage[top=1in,bottom=1in,left=1.25in,right=1.25in]{geometry}

\usepackage{todonotes}
\usepackage{tikz-cd}
\usepackage{stmaryrd}

\usepackage{relsize}

\usepackage[bbgreekl]{mathbbol}
\DeclareSymbolFontAlphabet{\mathbbl}{bbold}
\newcommand{\Prism}{{\mathlarger{\mathbbl{\Delta}}}}

\usepackage{comment}

\usepackage[breaklinks,colorlinks,linkcolor=blue,anchorcolor=blue,citecolor=blue,urlcolor=blue]{hyperref}

\newtheorem{theorem}{Theorem}[section]
\newtheorem{definition}[theorem]{Definition}
\newtheorem{remark}[theorem]{Remark}

\newtheorem{proposition}[theorem]{Proposition}

\newtheorem{lemma}[theorem]{Lemma}
\newtheorem{example}[theorem]{Example}

\def\Q{\mathbb{Q}}
\def\F{\mathbb{F}}
\def\R{\mathbb{R}}
\def\Z{\mathbb{Z}}
\def\A{\mathbb{A}}
\def\H{\mathcal{H}}

\def\G{\mathbb{G}}

\def\D{\mathcal{D}}
\def\E{\mathcal{E}}

\def\Fc{\mathcal{F}}
\def\Fl{\mathscr{F}\ell}

\def\M{\mathcal{M}}

\def\Ol{\mathcal{O}}
\def\P{\mathcal{P}}

\def\T{\mathcal{T}}

\def\V{\mathcal{V}}

\def\g{\mathfrak{g}}
\def\u{\mathfrak{u}}
\def\p{\mathfrak{p}}

\def\Isom{\mathrm{Isom}}
\def\Fil{\mathrm{Fil}}
\def\Hom{\mathrm{Hom}}
\def\Lie{\mathrm{Lie}}

\def\GL{\mathrm{GL}}

\def\Rep{\mathrm{Rep}}

\def\Sh{\mathrm{Sh}}

\def\J{\mathcal{J}}
\def\MF{\mathcal{MF}}
\def\X{\mathcal{X}}
\def\U{\mathcal{U}}

\def\dR{\mathrm{dR}}
\def\et{\mathrm{et}}

\def\Spf{\mathrm{Spf}}
\def\Spec{\mathrm{Spec}}
\def\Vect{\mathrm{Vect}}

\def\MIC{\mathrm{MIC}}
\def\Qcoh{\mathrm{QCoh}}
\def\gr{\mathrm{gr}}
\def\Sym{\mathrm{Sym}}
\def\Higgs{\mathrm{Higgs}}
\def\Verm{\mathrm{Verm}}

\def\lan{\langle}
\def\ran{\rangle}

\def\lra{\longrightarrow}
\def\ra{\rightarrow}
\def\ov{\overline}
\def\ul{\underline}
\def\wh{\widehat}
\def\wt{\widetilde}
\def\st{\stackrel}
\def\tr{\textrm}

\begin{document}

\title{De Rham $F$-gauges and Shimura varieties}
\author{Xu Shen}
\date{}

\address{Morningside Center of Mathematics, Academy of Mathematics and Systems Science, Chinese Academy of Sciences\\
	No. 55, Zhongguancun East Road\\
	Beijing 100190, China}
\address{University of Chinese Academy of Sciences\\Beijing 100049, China}
\bigskip
 
\email{shen@math.ac.cn}

\renewcommand\thefootnote{}

\renewcommand{\thefootnote}{\arabic{footnote}}

\begin{abstract}
We study $F$-gauges for de Rham cohomology of smooth algebraic varieties in characteristic $p$. Applying to good reductions of Shimura varieties of Hodge type, we recover the Ekedahl-Oort stratifications by constructing universal de Rham $F$-gauges with $G$-structure.
We also study the cohomology of de Rham $F$-gauges on these varieties. In particular, in the PEL type case and when the weights of the flat automorphic vector bundles are $p$-small, we determine the $F$-gauge structure on their de Rham cohomology by the associated dual BGG complexes.
		
\end{abstract}

\maketitle
\tableofcontents

\section{Introduction}
Shimura varieties are certain algebraic varieties which play a very important role in the Langlands program for number fields. These algebraic varieties admit very rich arithmetic and geometric structure. In this paper, we study the mod $p$ geometry and cohomology of these varieties by the theory of $F$-gauges, as initiated by Fontaine-Jannsen \cite{FJ}, Drinfeld \cite{Dri18, Dri20, Dri23}, Bhatt-Lurie \cite{BL1, BL2, Bhatt} et al. In fact, we mainly restrict to the simplest version of the theory: de Rham $F$-gauges, which are roughly enhanced coefficient objects for the de Rham cohomology of smooth algebraic varieties in characteristic $p$. This theory already reveals new information on the good reductions of Shimura varieties.\\

Let $p$ be a prime.
Kisin \cite{Kis} and Kim-Madapusi Pera \cite{KM} (for $p=2$) have constructed smooth integral canonical models for Shimura varieties of abelian type at $p$. For simplicity, we restrict to the Hodge type case. Let $X$ be the geometric special fiber over $k=\ov{\F}_p$ of the integral canonical model constructed in \cite{Kis, KM} of a Hodge type Shimura variety with hyperspecial level at $p$. Then by the works of Moonen-Wedhorn \cite{MW} (in the PEL type case, see also \cite{ViehmannWedhorn2013} and the references therein) and C. Zhang (in the Hodge type case, \cite{Zhang2018EO}), there is a basic stratification \[X=\coprod_{w\in {}^JW}X^w,\] called the Ekedahl-Oort stratification of $X$, generalizing previous works of Ekedahl, Oort, and Goren-Oort in the Siegel and Hilbert modular varieties case. Here the index set ${}^JW$ is certain subset of distinguished elements in the absolute Weyl group of $G$, the group defining the Shimura variety, which we consider as a reductive group over $\F_p$. The Ekedahl-Oort stratification has nice properties, which lead to better understanding on the arithmetic geometry of $X$, for example see \cite{ViehmannWedhorn2013, GoldringKoskivirta2019, WZ}. \\

To construct the Ekedahl-Oort stratification, one applies the theory of $F$-zips with additional structure developed by Moonen-Wedhorn \cite{MW} and Pink-Wedhorn-Ziegler \cite{PWZ11, PWZ15}, by
first constructing the universal $G$-zip of type $\mu$ (the Hodge cocharacter) over $X$, then showing that the induced morphism \[\zeta: X\ra G\tr{-Zip}^\mu\] from $X$  to the moduli stack $G\tr{-Zip}^\mu$ of $G$-zips of type $\mu$ is smooth. By \cite{PWZ11} section 6, there is a partial order $\preceq$ on $^J W$, and we have a homeomorphism of topological spaces $|G\tr{-Zip}^\mu|\simeq (^J W, \preceq)$. The Ekedahl-Oort stratum $X^w$ is defined as the fiber of $\zeta$ at $w$.\\

Let $S$ be a scheme over $\F_p$. An $F$-zip over $S$ consists of a tuple $(\E, C^\bullet, D_\bullet, \varphi)$, where 
\begin{itemize}
	\item $\E$ is a finite locally free $\Ol_S$-module (i.e. a vector bundle over $S$), 
	\item $C^\bullet$ is a decreasing filtration on $\E$, 
	\item $D_\bullet$ is an increasing filtration on $\E$, 
	\item $\varphi: Fr_S^\ast\gr_C\E\st{\sim}{\ra}\gr_D\E$ is an isomorphism of the associated graded vector bundles up to pullback by the absolute Frobenius $Fr_S$ of $S$. 
\end{itemize}
Such a structure arises naturally from the relative de Rham cohomology (with constant coefficient) of a proper smooth morphism $f: Y\ra S$, once we assume that \emph{the Hodge-de Rham spectral sequence degenerates and the relative Hodge cohomology sheaves are locally free}:  fix an integer $i\geq 0$ and set $\E=H^i_{\dR}(Y/S)$, let $C^\bullet$ be the Hodge filtration on $\E$,  $D_\bullet$ the conjugate filtration, and $\varphi$ the morphism induced by the Cartier isomorphism, cf. \cite{MW} section 7 for more details.
In particular, $F$-zips are closely related to the Dieudonn\'e theory of 1-truncated $p$-divisible groups (\cite{MW, PWZ15}). One has a natural extension of the notion of $F$-zips to $G$-bundles, known as $G$-zips. In the case of the mod $p$ Hodge type Shimura variety $X$, we  consider the relative de Rham cohomology of the abelian scheme \[\mathcal{A}\ra X,\] which is induced from a fixed Siegel embedding. This gives us an $F$-zip over $X$. To construct the universal $G$-zip, one takes care about the Hodge cycles, and makes trivialization of the $F$-zip with respect to the standard one constructed from the Siegel datum, for more details see \cite{Zhang2018EO}.\\

The motivation of this paper comes from the following question: what is the relation between the Ekedahl-Oort stratification of $X$ and the $F$-zip structure on the de Rham cohomology\footnote{In \cite{WZ} Wedhorn-Ziegler studied the cycle classes of Ekedahl-Oort strata in the Chow ring of $X^{tor}$ (smooth toroidal compactification of $X$), showing that they generate the same subring as that generated by the Chern classes of automorphic vector bundles. However, this does not address our question yet.} $H^i_{\dR}(X/k)$? Apparently,
both of them reflect some perspective of the mod $p$ Hodge structure of $X$. \\

Here, to formulate the question one has to be careful. First, to make sense of the $F$-zip structure on the de Rham cohomology $H^i_{\dR}(X/k)$, one has to prove that the Hodge-de Rham spectral sequence degenerates, which is nontrivial in characteristic $p$. Next, in the setting of Shimura varieties, it will be more natural to consider de Rham cohomology with coefficients. The theory of automorphic vector bundles provides plenty of flat vector bundles $(\V,\nabla)$ over $X$, coming from algebraic representations of $G$. Then, it will need furthermore efforts to make sense of the $F$-zip structure on the de Rham cohomology \[H^i_{\dR}(X/k, (\V,\nabla)),\] since in the previous work of Moonen-Wedhorn \cite{MW} only de Rham cohomology with trivial coefficients were considered. At this point we are naturally led to the notion of de Rham $F$-gauges, which are roughly enrichments of $F$-zips by adding an additional datum of a flat connection $\nabla$, with necessary compatibility conditions. It turns out the theory of de Rham $F$-gauges unifies and strengthens  the two structures on $X$: Ekedahl-Oort stratification and the de Rham cohomology.\\ 

Let $S$ be a \emph{smooth} scheme over $\F_p$. A de Rham $F$-gauge (see Definition \ref{def F-gauge}) over $S$ consists of a tuple $(\E, \nabla, C^\bullet, D_\bullet, \varphi)$, where 
\begin{itemize}
\item $(\E, C^\bullet, D_\bullet, \varphi)$ is an $F$-zip over $S$, 
\item $\nabla$ is a flat connection on $\E$, such that   it satisfies the Griffith transversality condition with respective to $C^\bullet$, 
\item the filtration $D_\bullet$ is horizontal with respective to $\nabla$, and the induced connection on $\gr_D\E$ has zero $p$-curvature,
\item let $\theta_\nabla=\gr_C\nabla$ be the graded Higgs field on $\gr_C\E$, and $\psi_\nabla: \gr_D\E\ra \gr_D\E\otimes Fr_S^\ast\Omega_S^1$ the morphism induced by the last condition on $p$-curvature, then we require the isomorphism $\varphi: Fr_S^\ast\gr_C\E\st{\sim}{\ra}\gr_D\E$ to be compatible with $\theta_\nabla$ and $\psi_\nabla$: it induces an isomorphism
$\varphi: Fr_S^\ast\theta_\nabla\simeq \psi_\nabla$.
\end{itemize}
By a classical theorem of Katz (\cite{Kat72} Theorem 3.2), such a structure also arises naturally from a proper smooth morphism $f: Y\ra S$, under the same condition on the degeneration of the  Hodge-de Rham spectral sequence and locally freeness of the Hodge cohomology sheaves. In particular, the natural connection $\nabla$ on $\E=H^i_{\dR}(Y/S)$ is given by the Gauss-Manin connection, $\theta_\nabla$ is then the Kodaira-Spencer map, and $\psi_\nabla$ the induced $p$-curvature map. See also \cite{OV} Remark 3.19 for some generalized versions of Katz's theorem.
On the other hand,
one can show that this explicit definition is equivalent to the modern stacky approach of Drinfeld \cite{Dri20} and Bhatt-Lurie \cite{Bhatt}, where a de Rham $F$-gauge can be defined as a vector bundle\footnote{One can consider more general perfect complexes or even general complexes in the stable $\infty$-derived category of quasi-coherent sheaves on the syntomic stack, as in these references. Here we only restrict to the simplest version of vector bundles.} over the special fiber of the syntomic stack $\ov{S^{syn}}$, cf. Theorem \ref{thm F-gauge via stack}. Another theory of de Rham $F$-gauges also appeared in the work of Fontaine-Jannsen (cf. \cite{FJ}, especially subsection 8.2). It would be interesting to explicitly compare the definition here and the version in \cite{FJ}.\\

Let $F\tr{-Gauge}_{dR}(S)$ be the category of de Rham $F$-gauges over $S$.
One can study the cohomology of a de Rham $F$-gauge $(\E, \nabla, C^\bullet, D_\bullet, \varphi)$ over a proper smooth morphism of smooth schemes $Y/S$: under the above condition that the Hodge-de Rham spectral sequence degenerates and the relative Hodge cohomology sheaves are locally free, the de Rham cohomology $H^i_{\dR}(Y/S, (\E,\nabla))$ carries a natural structure of a de Rham $F$-gauge over $S$, by combining the results of Katz \cite{Kat72} and Moonen-Wedhorn \cite{MW}, cf. Proposition \ref{prop coh F-gauge}. Here, we just mention that to study the conjugate filtration on the de Rham complex, we make use of the conjugate filtration $N^\bullet$ on the sheaf of PD differential operators $\D_{Y/S}$, cf. \cite{OV} subsection 3.4. The graded sheaf $\gr_N\D_{Y/S}$ is a canonical split Azumaya algebra, which essentially reduces the study of conjugate filtration on de Rham complexes to that on the associated $p$-curvature complexes, cf. Proposition \ref{prop Frob gauge complexes}.\\

If $\mathcal{S}_2/W_2$ is a lift of $S/k$ over the truncated Witt ring $W_2=W_2(k)$, we have another closely related category $\MF^\nabla_{[0,p-2]}(\mathcal{S}_2/W_2)$, the category of $p$-torsion Fontaine modules $(M, \nabla, \Fil, \Phi)$, introduced by Faltings in \cite{Fal} (as a generalization of the previous work of Fontaine-Laffaille in the case $S=\Spec\,k$). Here the precise definition of $(M, \nabla, \Fil, \Phi)$ is a little technical, in particular there is a ``strongly divisibility condition'' on the Frobenius $\Phi$ and the Hodge filtration $\Fil$. Using the Cartier transform functor constructed in \cite{OV}, Ogus-Vologodsky gave a reformulation of this category in loc. cit. subsection 4.6. The link between this theory and de Rham $F$-gauges is as follows, cf. Proposition \ref{prop MF dR gauge}. There is a natural functor
\[ \MF^\nabla_{[0,p-2]}(\mathcal{S}_2/W_2) \ra F\tr{-Gauge}_{dR}(S),\]
which is fully faithful, with essential image the de Rham $F$-gauges $(\E, \nabla, C^\bullet, D_\bullet, \varphi)$ of nilpotent level $\leq p-1$, i.e. $\psi_\nabla^{p-1}=0$. We note that this functor was already implicitly discussed in \cite{OV} subsection 4.6. Here, it deserves to note that contrary to the category of $p$-torsion Fontaine modules, we do not need any lift of  $S/k$ to define de Rham $F$-gauges.\\

We have also a similar notion of de Rham $F$-gauge with $G$-structure of type $\mu$ (or for simplicity: de Rham $F$-gauge with ($G,\mu$)-structure), for a reductive group $G$ over $\F_p$ with a cocharacter $\mu$. There is a natural forgetful functor \[F\tr{-Gauge}_{dR}^{G,\mu}(S)\ra G\tr{-Zip}^\mu(S)\] from the category of de Rham $F$-gauge with $G$-structure of type $\mu$ to the category of $G$-zips of type $\mu$ over $S$, which is an equivalence of categories if $S=\Spec\,k$. From here one sees that the natural morphism between the corresponding stacks induces a homeomorphism of underling topological spaces. \\

In a recent paper \cite{Dri23}, Drinfeld has studied the moduli stack\footnote{In Drinfeld's paper \cite{Dri23}, the cocharacter $\mu$ is defined over $\F_p$. But one checks that with necessary modifications, his construction works for $\mu$ defined over a finite extension of $\F_p$, as in the setting of \cite{PWZ11, PWZ15}. Moreover, the stack $BT_1^{G,\mu}$ is exactly introduced to study the mod $p$ geometry of Shimura varieties. } $BT_1^{G,\mu}$ of de Rham $F$-gauge with $G$-structure of type $\mu$, with $\mu$ \emph{minuscule}. He shows that this is a smooth algebraic stack of dimension zero, and the forgetful map \[BT^{G,\mu}_1\ra G\tr{-Zip}^\mu\] is an fppf gerbe banded by a finite group scheme $Lau_1^G$. The stack $BT^{G,\mu}_1$ should be the special fiber of a smooth algebraic stack living over $\Ol_E$, see \cite{Dri23} Conjecture C.3.1 (which has recently been proved, cf. \cite{GM} Theorem A).
If $G=\GL_h$ and $\{\mu\}$ is given by $(1^{d}, 0^{h-d})$ for some integer $0\leq d\leq h$, Drinfeld also conjectures that the stack $BT^{G,\mu}_1$ is isomorphic to the stack of 1-truncated $p$-divisible groups of height $h$ and dimension $d$, cf. \cite{Dri23} Conjectures 4.5.2 and 4.5.3. Recently, Gardner and Madapusi announced a theorem which says that there is an equivalence between the category of  truncated $p$-divisible groups of height $h$ and dimension $d$ over $S$ and the category $F\tr{-Gauge}_{dR}^{G,\mu}(S)$ (cf. \cite{GM} Conjecture 1 and the paragraph below it).  In particular, this equivalence implies Drinfeld's conjecture.\\

Back to Shimura varieties and let $S=X$.
The first main result of this paper is as follows. It allows us to make the link between de Rham $F$-gauges and Ekedahl-Oort stratification of $X$.
 \begin{theorem}[Theorem \ref{thm univ F-gauge}]
 	Let $X$ be the geometric special fiber of the integral canonical model of a Hodge type Shimura variety with hyperspecial level at $p$.
 	There is a natural de Rham $F$-gauge with $(G, \mu)$-structure over $X$, which induces the universal $G$-zip of type $\mu$ given by $\zeta$. The induced morphism $\xi: X\ra BT^{G,\mu}_1$ is smooth.
 	In other words, we have a commutative diagram of smooth morphisms of algebraic stacks over $k$:
 	\[ \xymatrix{ X\ar[r]^-\xi\ar[rd]_\zeta & BT^{G,\mu}_1\ar[d]\\
 		& G\tr{-Zip}^\mu.
 	}\]
 \end{theorem}
 This theorem can be viewed as a strengthen of the main result of \cite{Zhang2018EO}. Roughly, the natural de Rham $F$-gauge with $(G, \mu)$-structure over $X$ is found by adding the Gauss-Mannin connection to the universal $G$-zip of type $\mu$, which essentially comes from the universal abelian scheme over $X$. The smoothness of the morphism $\xi$ can be proved similarly as the smoothness of $\zeta$ (\cite{Zhang2018EO}). But this is more complicated. Roughly, we identify the tangent spaces of $X$ with those of the stack $BT^{G,\mu}_\infty$ (using the work of Ito \cite{Ito}, see also \cite{IKY}) and the fact that $BT^{G,\mu}_\infty\ra BT^{G,\mu}_1$ is smooth (\cite{GM}), where $BT^{G,\mu}_\infty=\varprojlim_nBT^{G,\mu}_n$ is the limit of the stack $BT^{G,\mu}_n$ introduced by Drinfeld (in \cite{Dri23} Appendix C) and studied by Gardner-Madapusi in \cite{GM}.
With some extra efforts, one can extend the above theorem to the abelian type case (cf. \cite{Lov,SZ} for some constructions of this type). Recently there have been works to construct a mixed characteristic version of the morphism $\xi$, which is certainly more difficult (see the introduction of \cite{GM} and \cite{IKY}).\\

We have also a Tannakian definition of de Rham $F$-gauge with $G$-structure. Thus by the above theorem, we have a functor \[\Rep\,G\ra F\tr{-Gauge}_{dR}(X).\] Therefore, for any $\lambda \in X^\ast(T)_+$ such that $V_\lambda$ is defined over $\F_p$, we have an enrichment of the flat vector bundle $(\mathcal{V}_\lambda, \nabla)$  into a de Rham $F$-gauge
$(\mathcal{V}_\lambda, \nabla, C^\bullet, D_\bullet, \varphi)$.
Consider the de Rham $F$-gauge associated to the dual representation $V_\lambda^\vee$.
In the second main result, we study the $F$-gauge (= $F$-zip) structure on the de Rham cohomology groups \[H_{\dR}^i(X, \mathcal{V}_\lambda^\vee):=H^i_{\dR}(X/k, (\V_\lambda^\vee,\nabla)).\] It turns out that if $\lambda$ satisfies a small condition with respect to $p$, this $F$-gauge structure is determined by the dual BGG complex\footnote{It is interesting to note that the set ${}^JW$ appears once again in the construction of the (dual) BGG complexes. The two appearances of ${}^JW$ (in EO stratification and dual BGG complexes) is in fact the original  motivation of this work.} $BGG(\mathcal{V}_\lambda^\vee)$, cf. \cite{Fal83, PT, MT, LP}, which is a basic construction in the setting of Shimura varieties and closely related to the de Rham complex $DR(\V_\lambda^\vee,\nabla)$.
\begin{theorem}[Theorem \ref{thm de Rham coh BGG}]
	Suppose that $X$ is of PEL type. For simplicity, assume moreover  that $X$ is proper, and the reductive group $G$ defined by the PEL datum is connected. Let $n=\dim\,X$. Assume\footnote{This additional condition, and the condition in part (1) that $\lambda$ is $p$-small with $|\lambda|<p$, are imposed to apply the constructions and results of \cite{LS1} to ensure the associated Hodge-de Rham spectral sequence is degenerate. These can be removed once one can prove the degeneracy by other methods. On the other hand, if we work less explicitly with complexes and derived categories, then we would not need any degeneracy results on the Hodge-de Rham spectral sequence.} moreover $2n< p$ and $\max(2,r_\tau)<p$ whenever there is a local factor $\tr{Sp}_{2r_\tau}$ appearing in the decomposition of $G_k$ arising from the PEL datum. Then we have
	\begin{enumerate}
		\item For any $0\leq i\leq 2n$ and $\lambda\in X^\ast(T)_+$ such that $\lambda$ is $p$-small with $|\lambda|<p$, $V_\lambda$ is defined over $\F_p$, there is a natural $F$-zip structure on $H_{\dR}^i(X, \mathcal{V}_\lambda^\vee)$, which is induced by the cohomology of the de Rham $F$-gauge $(\mathcal{V}_\lambda^\vee, \nabla, C^\bullet, D_\bullet, \varphi)$.
		\item For the above $\lambda$, the $F$-zip structure on $H_{\dR}^i(X, \mathcal{V}_\lambda^\vee)$ is determined by the dual BGG complex $BGG(\mathcal{V}_\lambda^\vee)$ as follows. Let $H\in X_\ast(T)$ be the element defined by the conjugacy class of $\mu$. For each $a\in \Z$,  by the construction of dual BGG complexes, we have
		\[
		H^i\Big(X, \gr_C^aBGG(\mathcal{V}_\lambda^\vee)\Big) =\bigoplus_{w\in {}^JW\atop w\cdot\lambda(H)=-a}H^{i-\ell(w)}(X, \mathcal{W}^\vee_{w\cdot\lambda}). \]
		Then there is a commutative diagram of isomorphisms
		\[\xymatrix{
			Fr_k^\ast \gr_C^aH^i_{\dR}(X,\mathcal{V}_\lambda^\vee)\ar[r]^{\varphi_a}& \gr_D^aH^i_{\dR}(X,\mathcal{V}_\lambda^\vee)\\
			Fr_k^\ast H^i\Big(X, \gr_C^aBGG(\mathcal{V}_\lambda^\vee)\Big)\ar[u]_{\sim} \ar[r]^{\varphi_a}&
			H^i\Big(X, \gr_D^aFr_X^\ast BGG(\mathcal{V}_\lambda^\vee)\Big)\ar[u]_{\sim}.
		}\] 
	\end{enumerate}
\end{theorem}
For the notion of $p$-small weights $\lambda$ and the notation $w\cdot \lambda$ and $|\lambda|$, see subsection \ref{subsection BGG} and \ref{subsection dR coh}. The sheaves $\mathcal{W}^\vee_{w\cdot\lambda}$ are also automorphic vector bundles on $X$, but without connection, see subsection \ref{subsection auto vb}.
Here we restrict to the PEL type case, in order to use the theory of Kuga families \cite{LS1} to deduce that the Hodge-de Rham spectral sequence is degenerate (see also \cite{Lan, LS2}). For the second part of the theorem, on the Hodge filtration side, it is already known that the inclusion of graded complexes \[\gr_CBGG(\V_\lambda^\vee)\hookrightarrow \gr_CDR(\V_\lambda^\vee,\nabla)\] is quasi-isomorphic by the key properties of dual BGG complexes. 
We observe that the $p$-curvature complex together with the conjugate filtration associated to $(\V_\lambda^\vee, \nabla)$ can be realized from a complex of Verma modules,  cf. Proposition \ref{prop p-curvature complex and p-std}. In particular, we apply the crystalline description of the $p$-curvature, see \cite{OV} Proposition 1.7, to achieve this construction.
This is similar to the case of de Rham complex with Hodge filtration, which can be constructed from the standard complex of Verma modules. Then we show the conjugate filtrations of the dual BGG complex and the de Rham complex are compatible, and the induced map between the graded complexes is a quasi-isomorphism.\\

We briefly review the structure of this article. In section 2, we recall various related notions on flat connections in characteristic $p$. In particular, we review the works of Ogus-Vologodsky \cite{OV} on non-abelian Hodge theory in characteristic $p$, the notion of $p$-torsion Fontaine modules of Faltings, as reformulated in \cite{OV}, and the theory of $F$-zips and $G$-zips due to Moonen-Wedhorn \cite{MW} and Pink-Wedhorn-Ziegler \cite{PWZ11, PWZ15}. In section 3, we study basic properties of de Rham $F$-gauges, their cohomology, and the theory with additional structure. In section 4, we apply the previous theory to good reductions of Shimura varieties of Hodge type. We first construct universal de Rham $F$-gauges with $G$-structure and show the smoothness of the induced map $\xi$.
Then we  study the cohomology of de Rham $F$-gauges on these varieties. In the PEL type case and when the weights of the flat automorphic vector bundles are $p$-small, we determine the $F$-gauge structure on their de Rham cohomology by the associated dual BGG complexes.\\
\\
\textbf{Acknowledgments.} We thank Heng Du, Shizhang Li, Mao Sheng, and Daxin Xu for some helpful conversations during the preparation of this work. We also thank Yu Min for valuable discussions. We sincerely thank the referee for helpful comments and suggestions, and for pointing out some inaccuracies in the previous version.
The author was partially supported by the National Key R$\&$D Program of China 2020YFA0712600, the CAS Project for Young Scientists in Basic Research, Grant No. YSBR-033, and the NSFC grant No. 12288201.

\section{Preliminaries}

Let $k$ be a perfect field of characteristic $p$ and $X/k$ a smooth $k$-scheme. In this section, we review some background on the non-abelian Hodge theory on $X$. We first fix some notations on Frobenius morphisms: let $Fr_X: X\ra X$ and $Fr_k: \Spec\,k\ra \Spec\,k$ be the absolute Frobenius maps, $\pi: X'\ra X$ the pullback of $X$ under $Fr_k$, and $F_{X/k}: X\ra X'$ the relative Frobenius map, so that $Fr_X=\pi_X\circ F_{X/k}$ and we have the following commutative digram
	\[\xymatrix{
		X \ar@/_/[ddr] \ar@/^/[drr]^{Fr_X} \ar[dr]|-{F_{X/k}}\\
		& X' \ar[r]^{\pi_X} \ar[d] & X\ar[d] \\
		& \Spec\,k\ar[r]^{Fr_k} & \Spec\,k,
	}\]
	where the square is cartesian. Later, we will also denote $X'=X^{(p)}$ for compatibility of notations. 
	In the following, most results hold more generally in the relative smooth setting for $X/S$, with $S$ a scheme over $\F_p$. For our later purpose, we will mainly restrict to the case $S=\Spec\,k$.
	
\subsection{PD differential operators and stratifications}\label{subsect PD diff}
Consider the diagonal embedding $X\hookrightarrow X\times_kX$ with ideal sheaf $I$. Then we have $\Omega_X^1=\Omega_{X/k}^1\simeq I/I^2$ as $\Ol_X$-modules.
Let $P_X=P_{X/k}$ denote the PD envelope of the diagonal embedding $X\hookrightarrow X\times_kX$ and $\P_X=\P_{X/k}$ its structure sheaf.  Let $J_X=J_{X/k}$ the ideal of $X$ in $P_{X}$.  Consider the completion of $\P_{X}$ with respect to the PD powers $J_{X}^{[n]}$: $\wh{\P}_{X}=\varprojlim_n\P_{X}/J_{X}^{[n]}$.  For any integer $n\geq 0$, denote \[\P_{X}^n=\P_{X}/J_{X}^{[n+1]}.\] For $n=1$, we have a natural isomorphism $\Omega_X^1\simeq J_X/J_X^{[2]}$ and a split exact sequence of $\Ol_X$-modules
\[0\lra \Omega_X^1\lra \P_{X}^1\lra \Ol_X\lra 0. \]

We recall the basic notions of PD differential operators and PD stratifications as in \cite{BO} section 4.
For $\Ol_X$-modules $\E$ and $\Fc$,  a PD differential operator $\E\ra \Fc$ of order $\leq n$ is an $\Ol_X$-linear map $\P^n_X\otimes\E\ra \Fc$, and an HPD differential operator $\E\ra \Fc$ is an $\Ol_X$-linear map $\P_X\otimes \E\ra \Fc$. A PD stratification on $\E$ over $\wh{\P}_X$ is a collection $(\varepsilon_n)_{n\geq 0}$ of isomorphisms
\[\varepsilon_n: \P^n_X\otimes\E\lra \E\otimes\P^n_X\]
such that $\varepsilon_0=\tr{id}_\E$, each $\varepsilon_n$ is $\P^n_X$-linear, when $n$ varies these $\varepsilon_n$ are compatible under the projections between $\P_X^n$, and for all $m$ and $n$, a cocycle condition holds; see \cite{BO} Definition 4.3 for the precise meaning of this condition. An HPD stratification on $\E$ over $\P_X$ is a $\P_X$-linear isomorphism \[\varepsilon: \P_X\otimes\E\lra \E\otimes\P_X\] such that $\varepsilon$ reduces to the identity mod $J_X$ and the cocycle condition holds.

Let $\D_X=\D_{X/k}$ be the sheaf of algebraic differential operators generated by $\Ol_X$ and the tangent sheaf $\T_X=\T_{X/k}$, subject to the module and commutator relations 
\begin{itemize}
	\item $f\cdot \partial =f\partial, \quad \partial\cdot f-f\cdot \partial=\partial(f), \quad \partial\in \T_X, f\in \Ol_X$,
	\item $\partial_1\cdot \partial_2-\partial_2\cdot\partial_1=[\partial_1, \partial_2], \quad \partial_1, \partial_2\in \T_X$. 
\end{itemize}
In other words, $\D_X$ is the envelope algebra of the tangent Lie algebroid $\T_X$. One can check that $\D_X$ coincides with the sheaf of PD differential operators on $\Ol_X$ as defined above. The sheaf $\D_X$ carries a natural increasing filtration $F_\bullet=\D_X^{\leq\bullet}$ by the degree of differential operators $\D_X=\bigcup_n\D_X^{\leq n}$, with $\D_X^{\leq 0}=\Ol_X, \D_X^{\leq n+1}=\D_X^{\leq n}+\T_X\cdot \D_X^{\leq n}$. As in the case of characteristic zero, one can check that the associated graded sheaf can be described as the symmetric algebra of $\T_X$: \[\gr_F\D_X\simeq \Sym\T_X.\]
The sheaf $\D_{X}$ is dual to $\wh{\P}_{X}$ in the sense that we have perfect pairing $\D_{X}\times \wh{\P}_{X}\ra \Ol_X$, which induces
an isomorphism of $\Ol_X$-modules
\[\D_{X}^{\leq n}\simeq \H om(\P_{X}^n,\Ol_X), \quad \D_{X}\simeq \varinjlim_n\H om(\P_{X}^n,\Ol_X),\]
where we use the left $\Ol_X$-module structure on $\P_{X}$.

\subsection{Flat vector bundles in characteristic $p$ and $p$-curvature}

Let $\E$ be a vector bundle over $X$, considered as a locally free sheaf of finite rank. Recall a connection $\nabla$ on $\E$ is given by a $k$-linear map of sheaves \[\nabla: \E\ra \E\otimes \Omega^1_{X},\] satisfying the Leibniz rule \[\nabla(fs)=f\nabla(s)+s\otimes df,\] for sections $f, s$ of $\E$ and $\Ol_X$ respectively. The connection $\nabla$  is called flat (or integrable), if \[\nabla_1\circ\nabla=0: \E\ra \E\otimes \Omega^2_{X},\]where $\nabla_1: \E\otimes \Omega^1_{X}\ra \E\otimes \Omega^2_{X}$ is the induced map given by $\nabla_1(e\otimes \omega)=e\otimes d\omega - \nabla(e)\wedge \omega$, for sections $e, \omega$ of $\E$ and $\Omega_{X}^1$ respectively.
In this case, we call $(\E,\nabla)$ a flat vector bundle. The datum of a connection $\nabla$ on $\E$ is equivalent to an $\Ol_X$-linear morphism of sheaves
\[\nabla: \T_{X}\ra \mathcal{E}nd_k(\E),\] and $\nabla$ is flat if and only if the map $\nabla: \T_{X}\ra \mathcal{E}nd_k(\E)$ is also a Lie algebra homomorphism, cf. \cite{Kat70} section 1.  We can generalize the above definition to more general quasi-coherent sheaves. Let $\MIC(X)$ be the category of quasi-coherent $\Ol_X$-modules with flat connections. We have the following equivalent description of this category.
\begin{theorem}[\cite{BO} Theorems 4.8 and 4.12]\label{thm flat connection}
Let $\E$ be a quasi-coherent $\Ol_X$-module. Then the following structures on $\E$ are equivalent:
\begin{enumerate}
\item a flat connection $\nabla: \E\ra \E\otimes\Omega_X^1$,
 \item a quasi-coherent left $\D_X$-module structure on $\E$,
\item a PD stratification on $\E$.
\end{enumerate}
The flat connection $\nabla$ is quasi-nilpotent (in the sense of \cite{BO} Definition 4.10)
if and only if the stratification is HPD.
\end{theorem}
In particular, we have
\[\MIC(X)=\Qcoh(\D_X).\] 
We view $\P_{X}$ as an $\Ol_X$-module from the right, which carries a canonical HPD stratification. Let $\nabla_\P$ be the integral connection on $\P_{X}$ corresponding to this stratification. Then $\nabla_\P(J_{X}^{[n]})\subset J_{X}^{[n-1]}\otimes\Omega^1_{X}$. Moreover $\nabla_\P$ induces an integral connection $\nabla_{\wh{\P}}$ on $\wh{\P}_{X}$.
\begin{remark}
The datum $(\E,\nabla)$ of a quasi-coherent $\Ol_X$-module together with a flat connection can be also described as a quasi-coherent sheaf over the de Rham stack $(X/k)^{dR}$ attached to $X/k$. This perspective is due to Simpson. See Theorem \ref{thm F-gauge via stack} for a similar statement for certain related objects with much richer structure.
\end{remark}

As we are working in characteristic $p$, there is a concept of $p$-curvature attached to each flat connection $(\E,\nabla)$, which plays a very important role in the theory. First note that if $\partial$ is a section of $\T_{X}$, then so is its $p$-th iterate $\partial^{(p)}$, which can be checked by viewing $\partial$ as a $k$-derivation on regular functions and using  the Leibniz rule. Now, for a flat vector bundle $(\E,\nabla)$, its $p$-curvature is defined as the map of sheaves
\[\psi_\nabla: \T_{X}\ra \mathcal{E}nd_k(\E),\quad \partial \mapsto \nabla(\partial)^p-\nabla(\partial^{(p)}). \] 
Therefore $\psi_\nabla=0$ if and only if the map $\nabla: \T_{X}\ra \mathcal{E}nd_k(\E)$ is a $p$-restricted Lie algebra homomorphism. In the general case, by \cite{Kat70} Proposition 5.2, $\psi_\nabla$ is $p$-linear. Therefore, as $Fr_X^\ast\T_X=F_{X/k}^\ast\T_{X'}$, the datum $\psi_\nabla$ is equivalent to  an $\Ol_X$-linear map \[\psi_\nabla: \E\ra \E\otimes_{\Ol_X} F_{X/k}^\ast \Omega^1_{X'}.\]
If $\dim\,X=n$ and $U\subset X$ is an open subset which is \'etale over $\A_k^n$ with coordinates $x_1,\dots, x_n$, then over $U$, the $p$-curvature $\psi_\nabla$ is determined by the endomorphisms of $\E|_U$: \[\psi_\nabla(\partial_1), \,\cdots,\, \psi_\nabla(\partial_n),\] where $\partial_i=\frac{\partial}{\partial x_i}$  $(1\leq i\leq n)$ is the induced base of $\T_X$ over $U$. Note that $\partial_i^{(p)}=0$ for all $1\leq i\leq n$, thus $\psi_\nabla(\partial_i)=\nabla(\partial_i)^p$.

Let $\Qcoh(X')$ be the category of quasi-coherent $\Ol_{X'}$-modules.
Recall the following classical Cartier descent:
\begin{theorem}[\cite{Kat70} Theorem 5.1]\label{thm cartier descent}
The functor $\H\mapsto (F_{X/k}^\ast\H, \nabla^{can})$ induces an equivalence of categories \[\Qcoh(X')\st{\sim}{\lra} \MIC(X)^{\psi_\nabla=0}, \] where $\MIC(X)^{\psi_\nabla=0}$ is the full subcategory of $\MIC(X)$  consists of $(\E, \nabla)$ with zero $p$-curvature. The quasi-inverse of the above functor is given by \[\MIC(X)^{\psi_\nabla=0}\lra \Qcoh(X'),\quad (\E, \nabla)\mapsto \E^{\nabla}.\]
\end{theorem}

In \cite{Kat70}, a flat connection $(\E, \nabla)$ on $X$ is called nilpotent of level $\leq N$, if the map \[\psi_\nabla^N: \E\ra \E\otimes (F^\ast_{X/k}\Omega^1_{X'})^{\otimes N}\] induced by the $p$-curvature $\psi_\nabla$ is zero. One can check that this notion is equivalent to the notion of quasi-nilpotent above (\cite{BO} Definition 4.10). By \cite{Kat70} Corollary 5.5, $(\E, \nabla)$ is nilpotent of level $N$ if and only if there exists a decreasing filtration of length $N$ \[\E=\Fil^0\supset \Fil^1\supset\cdots\supset\Fil^{N-1}\supset\Fil^N=0\] which is horizontal with respect to $\nabla$, such that the associated graded connection on $\gr_\Fil\E$ has zero $p$-curvature.
Flat connections of geometric origin are nilpotent:
\begin{example}[Katz, \cite{Kat70} Theorem 5.10]
	Let $f: Y\ra X$ be a proper smooth morphism between smooth algebraic varieties over $k$. Then for each $i\geq 0$, the relative de Rham cohomology $R^if_{dR\ast}\Ol_Y$ admits a canonical flat connection, called the Gauss-Manin connection, which is nilpotent.
\end{example}

Let $n=\dim\,X$.
For a flat vector bundle $(\E,\nabla)$, we have the de Rham complex \[DR(\E,\nabla)=(\E\otimes\Omega_{X}^\bullet, d^\bullet).\] More explicitly, this is the following complex of $\Ol_X$-modules
\[0\ra \E\st{d^0}{\lra} \E\otimes \Omega_{X}^1\st{d^1}{\lra} \E\otimes\Omega_{X}^2\st{d^2}{\lra} \cdots \st{d^{n-1}}{\lra} \E\otimes \Omega_{X}^n\lra 0, \]
where $d^0=\nabla$ and for $0\leq i\leq n$, \[d^i(e\otimes \omega)= e\otimes d\omega+ (-1)^i\nabla(e)\wedge \omega\] sections $e, \omega$ of $\E$ and $\Omega_X^i$ respectively. Note that although the differentials $d^i$ are not $\Ol_X$-linear, they are $F_{X/k}^{-1}\Ol_{X'}$-linear (note also $\Ol_X$ is  locally free of rank $p^n$ over $F_{X/k}^{-1}\Ol_{X'}$). As a result, the complex of $\Ol_{X'}$-modules $F_{X/k\ast}DR(\E,\nabla)$ is $\Ol_{X'}$-linear.
Following Grothendieck,
the de Rham cohomology of $X/k$ with coefficient in $(\E,\nabla)$ is defined by the hypercohomology of $DR(\E,\nabla)$
\[H^\ast_{\tr{dR}}(X/k,(\E,\nabla))=R^\ast\Gamma(X, DR(\E,\nabla)).\]
On the other hand, one can check that the $p$-curvature $\psi_\nabla$ satisfies an integral condition (cf. \cite{Ogus}): let $\psi^0=\psi$, and for $0\leq i\leq n$ let $\psi^i$ be the composition of \[\E\otimes F_{X/k}^\ast\Omega_{X'}^i\st{\psi\otimes id}{\lra} \E\otimes F_{X/k}^\ast\Omega_{X'}^1\otimes F_{X/k}^\ast\Omega_{X'}^i \st{id\otimes \wedge}{\lra} \E\otimes F_{X/k}^\ast\Omega_{X'}^{i+1},\]  then $\psi^i\circ \psi^{i-1}=0$.
Therefore,
we have the $p$-curvature complex
\[K(\E,\psi)=(\E\otimes F_{X/k}^\ast\Omega_{X'}^\bullet, \psi^\bullet)\] in the following form
\[0\ra \E\st{\psi^0}{\lra} \E\otimes F_{X/k}^\ast\Omega_{X'}^1\st{\psi^1}{\lra}  \E\otimes F_{X/k}^\ast\Omega_{X'}^2\st{\psi^2}{\lra}  \cdots \st{\psi^{n-1}}{\lra}  \E\otimes F_{X/k}^\ast\Omega_{X'}^n\lra 0.\]
Note that in case of $\psi_\nabla=0$, e.g. $(\E,\nabla)=(\Ol_X, d)$ the constant connection, the associated $p$-curvature complex has trivial differentials. Sometimes it is useful to consider the $\Ol_{X'}$-linear $p$-curvature complex $F_{X/k\ast}K(\E,\psi)$ on $X'$, which by the projection formula has the form
\[0\ra F_{X/k\ast}\E\st{\psi^0}{\lra} F_{X/k\ast}\E\otimes \Omega_{X'}^1\st{\psi^1}{\lra}  F_{X/k\ast}\E\otimes \Omega_{X'}^2\st{\psi^2}{\lra}  \cdots \st{\psi^{n-1}}{\lra}  F_{X/k\ast}\E\otimes \Omega_{X'}^n\lra 0.\]

Since $k$ is of characteristic $p$, the differentials of $F_{X/k\ast}\Omega^\bullet_{X}$ are $\Ol_{X'}$-linear. If $x$ is a local section of $\Ol_X$, then $d\pi_X^\ast(x)$ is a local section of $\Omega^1_{X'}$. One checks that there exists a unique homomorphism of $\Ol_{X'}$-modules \[C^{-1}: \Omega^1_{X'}\ra \H^1(F_{X/k\ast}\Omega^\bullet_{X})\] such that $C^{-1}(d\pi_X^\ast(x))$ is the class of $x^{p-1}dx$.
Recall the following classical theorem due to Cartier:
\begin{theorem}[\cite{Kat70} Theorem 7.2]\label{thm cartier isom}
	For all $i\geq 0$, there exists a unique isomorphism of $\Ol_{X'}$-modules
	\[C_i^{-1}: \Omega^i_{X'}\st{\sim}{\lra} \H^i(F_{X/k\ast}\Omega^\bullet_{X}) \]
	such that
	\[C_0^{-1}(1)=1,\quad C_1^{-1}=C^{-1},\quad C_{i+j}^{-1}(\omega\wedge\omega')=C_i^{-1}(\omega)\wedge C_j^{-1}(\omega'),\]
	for local sections $\omega$ and $\omega'$ of $\Omega^i_{X'}$ and $\Omega^j_{X'}$ respectively.
\end{theorem}
We discuss a little more about the inverse Cartier isomorphisms.
First, for each $i$, the original inverse Cartier isomorphism \[C_i^{-1}: \Omega^i_{X}\st{\sim}{\lra} \H^i(\Omega^\bullet_{X})\] is $p$-linear (cf. \cite{Kat72} 7.1.4 for some discussions on the original Cartier isomorphism and the version above). Therefore it is equivalent to an $\Ol_X$-linear isomorphism
\[C_i^{-1}: \Omega^i_{X}\st{\sim}{\lra} \H^i(Fr_{X\ast}\Omega^\bullet_{X}).\] Next,
as $\Omega^i_{X'}=\pi_X^\ast\Omega^i_{X}$, we can rewrite $C_i^{-1}$ as \[\pi_X^\ast\Omega^i_{X}\st{\sim}{\lra} \H^i(F_{X/k\ast}\Omega^\bullet_{X}).\] We have the natural decreasing Hodge filtration $C^\bullet$ on the de Rham complex $\Omega^\bullet_{X}$ with
\[C^i(\Omega^\bullet_{X})^j=\begin{cases} 0, & j< i\\
\Omega^j_{X}, & j\geq i \end{cases}\] and
 \[\gr^i_{C}\Omega^\bullet_{X}=\Omega^i_{X}[-i].\] On the other hand, we have also
the increasing conjugate filtration $D_\bullet$ on $F_{X/k\ast}\Omega^\bullet_{X}$ with
\[ D_i(F_{X/k\ast}\Omega^\bullet_{X})^j=\tau_{\leq i}(F_{X/k\ast}\Omega^\bullet_{X})^j=\begin{cases}
F_{X/k\ast}\Omega^j_{X}, & j< i\\
\ker( F_{X/k\ast}\Omega^i_{X} \ra F_{X/k\ast}\Omega^{i+1}_{X} ),& j=i\\
0, & j>i
\end{cases} \] and
 \[\gr^i_{D}F_{X/k\ast}\Omega^\bullet_{X}=\H^i(F_{X/k\ast}\Omega^\bullet_{X})[-i].\] 
 The  inverse Cartier isomorphism $C_i^{-1}$ can be rewritten as
\[C_i^{-1}:  \pi_X^\ast\gr^i_{C}\Omega^\bullet_{X} \st{\sim}{\lra} \gr^i_{D}F_{X/k\ast}\Omega^\bullet_{X}.\]

\subsection{The Azumaya algebra property and the conjugate filtration of $\D_X$}\label{subsection conj fil}

In characteristic $p$, the sheaf $\D_X$ has a large center $Z(\D_X)$ (contrary to the characteristic zero setting). To study the center, sometimes it is convenient to consider $F_{X/k\ast}\D_X$ instead of $\D_X$. Note that as $F_{X/k}$ is affine, the direct image functor $F_{X/k \ast}$ induces an equivalence of categories: \[\Qcoh(\D_X)\simeq \Qcoh(F_{X/k\ast}\D_X).\]
A key observation of Bezrukavnikov-Mirkovi\'c-Rumynin is
\begin{theorem}[\cite{BMR}]\label{thm center D}
\begin{enumerate}
\item
The $p$-curvature map $\psi: \T_{X'}\ra F_{X/k\ast}\D_X, v\mapsto v^p-v^{(p)}$ factors through the center $Z(F_{X/k\ast}\D_X)$ of $F_{X/k\ast}\D_X$, and induces an isomorphism
\[\Sym \T_{X'} \simeq Z(F_{X/k\ast}\D_X). \]
\item The sheaf $F_{X/k\ast}\D_X$ is an Azumaya algebra of rank $p^{2\dim X}$ over $\Sym \T_{X'}$.
\end{enumerate}
\end{theorem}

By adjunction, we have then
\[Fr_X^\ast\Sym \T_{X}=F_{X/k}^\ast\Sym\T_{X'} \simeq Z_{\D_X}(\Ol_X)\hookrightarrow\D_X, \]
where $Z_{\D_X}(\Ol_X)$ is the centralizer of $\Ol_X$ in $\D_X$, and for the second isomorphism see \cite{BMR} section 2. One checks easily that $\D_X$ is locally free of rank $p^{\dim X}$ over $Z_{\D_X}(\Ol_X)$. The proof of Theorem \ref{thm center D} (2) actually shows that 
\[F_{X/k}^\ast F_{X/k\ast}\D_X\simeq \E nd_{Z_{\D_X}(\Ol_X)}(\D_X). \]
Following \cite{OV} section 3.4, we introduce a decreasing filtration $N^\bullet$ on $\D_X$ by ideals
\[\cdots \subset \mathcal{J}_X^i\subset \cdots \subset \mathcal{J}_X\subset  \D_X,\]
where \[\mathcal{J}_X= \mathcal{J}_X^1=(F_{X/k}^\ast\Sym^1\T_{X'})\cdot\D_X,\quad \mathcal{J}_X^i=(F_{X/k}^\ast\Sym^i\T_{X'})\cdot\D_X\] are ideals of $\D_X$ generated by the corresponding one of $F_{X/k}^\ast\Sym\T_{X'} $. More precisely, this filtration is induced from the deceasing filtration $N^\bullet$ on $F_{X/k}^\ast\Sym \T_{X'}$ by ideals \[\cdots \subset \mathscr{J}_X^i\subset \cdots \subset \mathscr{J}_X\subset  F_{X/k}^\ast\Sym\T_{X'},\quad \tr{with}\quad \mathscr{J}_X^i:=F_{X/k}^\ast\Sym^i\T_{X'}.\] Then we have \[\gr_{N}F_{X/k}^\ast\Sym \T_{X'}\simeq F_{X/k}^\ast\Sym \T_{X'}.\] The natural inclusion \[F_{X/k}^\ast\Sym \T_{X'}\hookrightarrow \D_X\] given by $p$-curvature preserves the filtrations by construction. 
\begin{proposition}
The associated graded $\gr_N\D_X$ can be described as\footnote{Recall that for the order filtration $F_\bullet$ on $\D_X$, the associated graded $\gr_F\D_X\simeq \Sym \T_X$.}
\[\gr_N\D_X\simeq (\D_X/\J_X)\otimes_{\Ol_{X}}F_{X/k}^\ast\Sym\T_{X'},\]
which is a canonically split tensor Azumaya algebra, with \[F_{X/k}^\ast F_{X/k\ast}\Ol_X\otimes_{\Ol_{X}}F_{X/k}^\ast\Sym\T_{X'}\] the graded splitting module.
\end{proposition} 
\begin{proof}
This follows from \cite{OV} the paragraph below Lemma 3.18. In fact, by \cite{Bert} Proposition 2.2.7, we have \[\D_X/\J_X\simeq \E nd_{F_{X/k}^{-1}\Ol_{X'}}(\Ol_X)\simeq \E nd_{\Ol_X}(F_{X/k}^\ast F_{X/k\ast}\Ol_X).\]
Therefore, we have an isomorphism 
\[(\D_X/\J_X)\otimes_{\Ol_{X}}F_{X/k}^\ast\Sym \T_{X'}\simeq \E nd_{F_{X/k}^\ast\Sym \T_{X'}}\Big(F_{X/k}^\ast F_{X/k \ast}\Ol_X\otimes_{\Ol_{X}}F_{X/k}^\ast\Sym\T_{X'}\Big). \]
\end{proof}
Hence we get an equivalence of categories
\[\begin{split}
C_{X}:\Qcoh(\gr_N\D_X)&\st{\sim}{\lra}\Qcoh(F_{X/k}^\ast\Sym \T_{X'}), \\ \M &\mapsto \H om_{\gr_N\D_X}\Big(F_{X/k}^\ast F_{X/k\ast}\Ol_X\otimes_{\Ol_{X}}F_{X/k}^\ast\Sym \T_{X'}, \M\Big).
\end{split} \] The quasi-inverse $C_X^{-1}: \Qcoh(F_{X/k}^\ast\Sym \T_{X'}) \lra \Qcoh(\gr_N\D_X)$ is given by \[\E\mapsto \E\otimes_{F_{X/k}^\ast\Sym \T_{X'}}\Big(F_{X/k}^\ast F_{X/k\ast}\Ol_X\otimes_{\Ol_{X}}F_{X/k}^\ast\Sym\T_{X'}\Big)\simeq \E\otimes_{\Ol_X}F_{X/k}^\ast F_{X/k\ast}\Ol_X.\]
In fact, we have an equivalence of derived categories
\[C_X: D\Big(Mod^\bullet(\gr_N\D_X)\Big)\st{\sim}{\lra}D\Big(Mod^\bullet(F_{X/k}^\ast\Sym \T_{X'})\Big),\]
where $D\Big(Mod^\bullet(\gr_N\D_X)\Big)$ is the derived category of graded 
$\gr_N\D_X$-modules, and similarly for $D\Big(Mod^\bullet(F_{X/k}^\ast\Sym \T_{X'})\Big)$.

\subsection{Higgs bundles and flat vector bundles}

Let $Y/k$ be a smooth scheme.
A Higgs sheaf on $Y$ is defined by a pair $(\H, \theta)$, where $\H$ is a quasi-coherent $\Ol_{Y}$-module, \[\theta: \H\ra \H\otimes_{\Ol_Y}\Omega^1_{Y}\] is a morphism of $\Ol_Y$-modules, satisfying $\theta \wedge \theta =0$. As in the case of flat connections, a Higgs sheaf can be equivalently given by a map \[\theta: \T_Y\ra \mathcal{E}nd_{\Ol_Y}(\H),\] such that the image of $\theta$ is a commuting subsheaf of operators. We denote by $\Higgs(Y)$ the category of Higgs sheaves on $Y$.

Consider the sheaf of algebra $\wh{\Gamma}\Omega_{Y}^1$, the complete PD algebra associated to $\Omega_{Y}^1$, which can be viewed as the structure sheaf of the PD formal completion $T_{PD}^\ast Y$ of the tangent bundle $T^\ast Y\ra Y$ along the zero section. We have a perfect pairing \[ \Sym\T_Y\times\wh{\Gamma}\Omega_{Y}^1\lra \Ol_Y.\]
Similar to Theorem \ref{thm flat connection}, we have
\begin{theorem}\label{thm higgs}
Let $\H$ be a quasi-coherent $\Ol_Y$-module. Then the following structures on $\H$ are equivalent:
\begin{enumerate}
	\item a Higgs filed $\theta: \H\ra \H\otimes\Omega_Y^1$,
	\item a $\Sym\T_Y$-module structure  on $\H$, 
	\item a PD stratification of $\H$ over $\wh{\Gamma}\Omega_{Y}^1$.
\end{enumerate}
\end{theorem}
The equivalence $(1)\Leftrightarrow(2)$ is classical. We omit the proof of the equivalence $(2)\Leftrightarrow(3)$ (which can be proved similarly as Theorem \ref{thm flat connection}; see also \cite{Oy} Theorem 1.2.10 and its proof for a slightly variant version) as actually we will not use it in the following.
 Let $q: T^\ast Y\ra Y$ be  the cotangent bundle , which is the geometric realization of $\Omega_Y^1=\H om(\T_Y, \Ol_Y)$, and we have \[q_\ast\Ol_{T^\ast Y}\simeq \Sym\T_Y.\] As $q$ is affine, the functor $q_\ast$ induces an equivalence of categories:
\[\Qcoh(T^\ast Y)\simeq \Qcoh(\Sym\T_Y)=\Higgs(Y). \]

Replacing the above $\Omega_Y^1$ and $\T_Y$ by a general locally free sheaf of finite rank $\Omega$ and its dual $\T$, we have the notion of $\T$-Higgs fields and $\T$-Higgs bundles. Back to our fixed smooth scheme $X/k$ and let $Y=X$. 
In particular, considering $\Omega=F_{X/k}^\ast\Omega^1_{X'}$ and $\T=F_{X/k}^\ast\T_{X'}$, we denote the associated category of twisted Higgs bundles as $F\tr{-}\Higgs(X)$. Then we have the $p$-curvature functor
\[\MIC(X)\lra F\tr{-}\Higgs(X), \quad (\E,\nabla)\mapsto (\E,\psi_\nabla). \]
As in Theorems \ref{thm flat connection} and \ref{thm higgs}, the category $F\tr{-}\Higgs(X)$ can be described as
\begin{theorem}
	Let $\E$ be a quasi-coherent $\Ol_X$-module.  Then the following structures on $\E$ are equivalent: 
	\begin{enumerate}
		\item a $F_{X/k}^\ast\T_{X'}$-Higgs field $\psi: \E\ra \E\otimes  F_{X/k}^\ast\T_{X'}$,
		\item a $F_{X/k}^\ast\Sym\T_{X'}$-module structure on $\E$,
		\item a PD stratification over the complete PD algebra
		$F_{X/k}^\ast\wh{\Gamma}\Omega^1_{X'}$.
	\end{enumerate}
\end{theorem}
Thus we have an equivalence of categories
\[F\tr{-}\Higgs(X)\simeq \Qcoh(F_{X/k}^\ast\Sym\T_{X'}).\]  Together with the equivalence $\MIC(X)\simeq \Qcoh(\D_X)$ by Theorem \ref{thm flat connection}, the $p$-curvature functor 
$\MIC(X)\lra F\tr{-}\Higgs(X)$ is given by the natural restriction functor by the inclusion \[F_{X/k}^\ast\Sym\T_{X'}\simeq Z_{\D_X}(\Ol_X)\subset \D_X.\] 
On the other hand, as $\Higgs(X')\simeq \Qcoh(\Sym\T_{X'})$ by Theorem \ref{thm higgs}, we have the natural pullback functor
\[\Higgs(X')\lra F\tr{-}\Higgs(X)\]
induced by $F_{X/k}^\ast$. 

Recall $I$ is the ideal sheaf of the diagonal embedding $X\hookrightarrow X\times_kX$. We have a natural map \[I\ra \P_X/I\P_X\] which is $p$-linear and zero on $I^2$, thus we get an induced map \[ Fr_X^\ast\Omega_X^1=F_{X/k}^\ast \Omega_{X'}^1\lra  \P_{X}/I\P_{X}, \quad \tr{and thus}\quad F_{X/k}^\ast\wh{\Gamma}\Omega^1_{X'}\lra \wh{\P}_{X}/I\wh{\P}_{X}.\]
\begin{proposition}\label{prop p-curvature}
\begin{enumerate}
\item The map $F_{X/k}^\ast \Omega_{X'}^1\lra  \P_{X}/I\P_{X}$ induces an isomorphism \[F_{X/k}^\ast \Omega_{X'}^1\simeq J_X/(J_X^{[p+1]}+ I\P_X). \]
\item The above map is an isomorphism of sheaves of algebras
\[ F_{X/k}^\ast\wh{\Gamma}\Omega^1_{X'} \simeq \wh{\P}_{X}/I\wh{\P}_{X}. \]
\item Under the perfect pairing $\D_{X}\times \wh{\P}_{X}\ra \Ol_X$, the induced quotient map \[\wh{\P}_{X}\ra \wh{\P}_{X}/I\wh{\P}_{X}\simeq F_{X/k}^\ast\wh{\Gamma}\Omega^1_{X'}\] is dual to the inclusion
$F_{X/k}^\ast\Sym\T_{X'} \hookrightarrow \D_{X}$ given by the $p$-curvature map
\[F_{X/k}^\ast\Sym\T_{X'} \simeq Z_{\D_X}(\Ol_X)\hookrightarrow \D_X.\]
\end{enumerate}
\end{proposition}
\begin{proof}
(1) is \cite{OV} Proposition 1.6. (2) follows from (1) and a local computation, for more details, see \cite{GSQ} Proposition 3.3 or \cite{Men} Proposition 7.6. Finally, (3) follows easily from (2), cf. \cite{Men} Lemma 7.8.
\end{proof}

 If $(\E,\nabla)\in \MIC(X)$ with the associated stratification isomorphism $\varepsilon: pr_2^\ast\E\ra pr_1^\ast\E$ (by Theorem \ref{thm flat connection}), where $pr_i: \wh{P}_X=\Spec\,\wh{\P}_X\ra X$ are the natural projections, then under the isomorphism in Proposition \ref{prop p-curvature} (1), by a result of Mochizuki, the $p$-curvature $\psi_\nabla: \E\ra \E\otimes F_{X/k}^\ast\T_{X'}$ is given by \[\psi_\nabla(e)=\varepsilon(pr_2^\ast(e))- pr_1^\ast(e)\quad \tr{mod}\quad J_X^{[p+1]}+ I\P_X\] for any section $e$ of $\E$, see \cite{OV} Proposition 1.7. By Proposition \ref{prop p-curvature} (3), this is compatible with the description of $p$-curvature by restriction of scalar along the injection
 $F_{X/k}^\ast\Sym\T_{X'} \hookrightarrow \D_X$.

Back to the setting of the beginning of this subsection. Consider the smooth scheme $Y/k$.
Let $N\geq 1$ be an integer and $(\H,\theta)$ a Higgs sheaf on $Y$. We say $\theta$ is nilpotent of level $\leq N$ if the induced morphism \[\theta^N: \H\ra \H\otimes (\Omega_Y^1)^{\otimes N}\] is zero. The resulting category is denoted by $\Higgs_N(Y)$. 
Now, we review the basic link between flat connections and Higgs sheaves, following \cite{OV}. 
Back to our fixed smooth scheme $X/k$ as in the beginning of this section, and let $Y=X'$. One basic version of the main results of Ogus-Vologodsky is as follows.
\begin{theorem}[\cite{OV} Theorem 2.8]\label{thm OV}
	A lifting $\mathcal{X}/W_2(k)$ of $X/k$ gives rise to an equivalence of categories
	\[C^{-1}_{\mathcal{X}/W_2(k)}: \Higgs_{p-1}(X')\st{\sim}{\lra} \MIC_{p-1}(X), \]
	which extends the equivalence of Cartier descent in Theorem \ref{thm cartier descent}. 
\end{theorem}
Very roughly, this theorem was proved in \cite{OV} by showing the Azumaya algebra $F_{X/k\ast}\D_X$ splits when pulling back along \[\gamma: T_{PD}^\ast{X'}\ra T^\ast X',\] where we view $F_{X/k\ast}\D_X$ as a sheaf over $T^\ast X'$ by viewing it as a module over $Z(F_{X/k\ast}\D_X)\simeq \Sym\T_{X'}$, and \[T_{PD}^\ast{X'}=\Spec_{\Sym\T_{X'}}\wh{\Gamma}\T_{X'}\] with the complete PD algebra $\wh{\Gamma}\T_{X'}$ over $\Sym\T_{X'}$.
We refer to \cite{OV} for more comprehensive and extensive study of the correspondence between Higgs sheaves and flat connections, including the relative version, the comparison between cohomology and the functoriality.
\begin{remark}
	\begin{enumerate}
		\item There are some other constructions of the equivalence above: \begin{itemize}
			\item Lan-Sheng-Zuo \cite{LSZ1} constructed the Cartier and inverse Cartier transforms by exponential twisting of the classical Cartier descent equivalences as in Theorem \ref{thm cartier descent},
			\item Oyama \cite{Oy} constructed the equivalences by introducing a Higgs site which is closely related to the crystalline site; moreover this approach also deduced the comparison theorem between de Rham cohomology and Higgs cohomology.
		\end{itemize}
		\item The work of Schepler \cite{Schep} extends the Ogus-Vologodsky equivalence to the logarithmic setting, which also contains a comparison theorem between de Rham cohomology and Higgs cohomology.
	\end{enumerate}
\end{remark}

\subsection{$p$-torsion Fontaine modules}
Besides flat connections and Higgs fields, we will need more structures: Hodge filtration and Frobenius. One version of such objects is given by the Fontaine modules (some other terminologies: relative Fontaine-Laffaille modules, or filtered $F$-crystals) introduced by Faltings in \cite{Fal}. In fact, we will only need a $p$-torsion version, as discussed in loc. cit. subsections II. c) and d). We will follow \cite{LSZ2} section 2 to give explicit descriptions.

Fix a lifting $\mathcal{X}_2/W_2$ of $X/k$. A $p$-torsion Fontaine module on $X$ consists of tuples $(M, \nabla, \Fil, \Phi)$, where
\begin{itemize}
	\item $(M, \nabla)$ is a flat coherent $\Ol_X$-module,
	\item $\Fil$ is a decreasing filtration on $M$ of length\footnote{One should be careful about the length. We follow the notation of \cite{Fal}, where there is also a larger category $\MF^\nabla_{[0,p-1]}(\X_2/W_2)$ of similar objects, but with length $\leq p$, i.e. $\Fil^p=0$. However, it is the category  $\MF^\nabla_{[0,p-2]}(\X_2/W_2)$ that is compatible with the theory in \cite{OV}. } $\leq p-1$, i.e. $\Fil$ is given by  \[M=\Fil^0\supset \Fil^1\supset\cdots\supset\Fil^{p-2}\supset\Fil^{p-1}=0,\] such that the Griffiths transversality condition holds: $\nabla(\Fil^i)\subset \Fil^{i-1}\otimes\Omega^1_{X/k}$,
	\item $\Phi$ is the Frobenius structure on $M$ given by the following data. Set $\wt{M}=\gr_{\Fil}M=\bigoplus_{i=0}^{p-2}\Fil^i/\Fil^{i+1}$. Then \[\Phi: Fr_X^\ast\wt{M}\st{\sim}{\lra} M\] is an isomorphism of $\Ol_X$-modules. Moreover, $\Phi$ is required to be compatible with $\Fil$  and $\nabla$ in the following sense: take an open cover $U$ of $X$, $\U\subset \X_2$ lifts of $U$, $Fr_\U: \U\ra \U$ lifts of the absolute Frobenius on $U$,  and
	\begin{itemize}
		\item let $\wt{\nabla}=\gr_\Fil\nabla$ be the graded Higgs field attached to $\nabla$ and $\Fil$ on $\wt{M}$;
		\item for each $U$, let $Fr_U^\ast\wt{\nabla}$ be the connection on $Fr_U^\ast\wt{M}_{U}$ defined by the formula:
		\[Fr_U^\ast\wt{\nabla}(f\otimes e)=df\otimes e+ f\cdot (\frac{dFr_\U}{p}\otimes 1)(1\otimes \wt{\nabla}(e)), \quad f\in \Ol_U, e\in M_{U},\] then the following diagram commutes:
		\[\xymatrix{ Fr_U^\ast \wt{M}_U\ar[r]^{\Phi_U} \ar[d]_{Fr_U^\ast\wt{\nabla}} & M_U\ar[d]^\nabla\\
			Fr_U^\ast \wt{M}_U\otimes \Omega_U^1\ar[r]^{\Phi_U\otimes id} & M_U\otimes\Omega_U^1.
	} \]
	\end{itemize}
	By \cite{Fal}, $M$ is in fact locally free, and $\Phi$ is well defined, i.e. it does not depend on the choice of local liftings. 
\end{itemize}
We denote the category of $p$-torsion Fontaine modules on $X$ by \[\MF^\nabla_{[0,p-2]}(\X_2/W_2).\] By \cite{Fal}, this is an abelian category.
If $\X/W$ is a smooth lifting of $X/k$, we have also the usual category of Fontaine modules $\MF^\nabla_{[0,p-2]}(\X/W)$ over $\X$, cf. \cite{Fal} p. 34-35. There is an obvious truncation functor \[\MF^\nabla_{[0,p-2]}(\X/W)\ra \MF^\nabla_{[0,p-2]}(\X_2/W_2).\] Moreover, one can check that $\MF^\nabla_{[0,p-2]}(\X_2/W_2)$ is equivalent to the full subcategory of $p$-torsion objects in $\MF^\nabla_{[0,p-2]}(\X/W)$, see  \cite{OV} the paragraph after Definition 4.16.

Using Theorem \ref{thm OV}, we can give an equivalent formulation of $p$-torsion Fontaine modules. Recall  the equivalences
\[\xymatrix{\Higgs_{p-1}(X') \ar@/^/[rr]^{C^{-1}_{\X_2/W_2}}& & \MIC_{p-1}(X)  \ar@/^/[ll]^{C_{\X_2/W_2}}. } \]
For an object $(M, \nabla, \Fil, \Phi)\in \MF^\nabla_{[0,p-2]}(\X_2/W_2)$, we get the associated Higgs bundle
\[(\H, \theta):=\pi_\ast\circ C_{\X_2/W_2}(M,\nabla)\in \Higgs_{p-1}(X),\] where recall $\pi: X'\ra X$ is the projection from the Frobenius twist.
As in the characteristic zero case, this Higgs bundle can be obtained more directly as
\[(\H, \theta)\simeq (\gr_{\Fil}M, \gr_{\Fil}\nabla). \]
More precisely, we have \[\H=\bigoplus_i\H_i,\quad \theta=\bigoplus_i\theta_i\] with \[\H_i=\gr_{\Fil}^iM=\Fil^i/\Fil^{i-1}\quad \tr{and}\quad \theta_i=\gr_{\Fil}\nabla^i: \Fil^i/\Fil^{i-1}\lra (\Fil^{i-1}/\Fil^{i-2})\otimes\Omega_X^1\] induced by the Griffiths transversality condition.
\begin{proposition}\label{prop p-torsion FM}
\begin{enumerate}
\item We have an equivalence of categories:
\[ \MF^\nabla_{[0,p-2]}(\X_2/W_2)\st{\sim}{\lra}\{(M,\nabla,\Fil,\phi)\},\] where the category on the right hand side is given by \cite{OV} Definition 4.16. Concretely, for an object $(M,\nabla,\Fil,\phi)$ in this category, the first three terms $(M,\nabla,\Fil)$ are the same as above,
the last term is an isomorphism \[\phi: C_{\X_2/W_2}^{-1}\circ \pi^\ast (\gr_{\Fil}M, \gr_{\Fil}\nabla)\st{\sim}{\lra} (M,\nabla).\] The equivalence is given by a natural functor $(M, \nabla, \Fil, \Phi)\mapsto (M, \nabla, \Fil, \phi)$.
\item We have an equivalence of categories:
\[ \MF^\nabla_{[0,p-2]}(\X_2/W_2)\st{\sim}{\lra}\{(\H,\theta,\Fil,\phi)\},\] where the category on the right hand side is the category of 1-periodic Higgs-de Rham flow introduced in \cite{LSZ2}. It consists of objects with
$(\H,\theta)\in\Higgs_{p-1}(X)$,
 $\Fil$ is a Griffiths transverse filtration on $C^{-1}_{\X_2/W_2}(\H,\theta)$, and $\phi$ is an isomorphism \[\phi: \gr_{\Fil}\circ C_{\X_2/W_2}^{-1}\circ \pi^\ast (\H, \theta)\st{\sim}{\lra} (\H,\theta).\]
 The equivalence is given by a natural functor 
$(M, \nabla, \Fil, \Phi)\mapsto (\gr_{\Fil}M, \gr_{\Fil}\nabla, \Fil, \phi)$.
\end{enumerate}
\end{proposition}
\begin{proof}
The first statement is contained in \cite{OV} section 4.6.
The second equivalence is proved in details in \cite{LSZ2} Proposition 3.3.
\end{proof}

\subsection{$F$-zips}\label{subsection F-zip}

We recall the notion of $F$-zips introduced by Moonen-Wedhorn in \cite{MW} (see also \cite{Wed}). Let $S$ be a scheme over $\F_p$ (not necessarily smooth). For an object $M$ over $S$, in the following we denote $M^{(p)}=Fr_S^\ast M$, the pullback of $M$ under the absolute Frobenius morphism $Fr_S: S\ra S$.
An $F$-zip over $S$ is a tuple $(M, C^\bullet, D_\bullet, \varphi)$ where
	\begin{itemize}
		\item $M$ is a locally free sheaf of finite rank on $S$,
		\item $C^\bullet$ is a descending filtration on $M$,
		\item $D_\bullet$ is an ascending filtration on $M$,
		\item $\varphi=\oplus_i\varphi_i: (\gr_CM)^{(p)}\st{\sim}{\lra}\gr_DM$ is an isomorphism of graded $\Ol_S$-modules: for each $i$, $\varphi_i: (C^i /C^{i+1})^{(p)}  \to D_i / D_{i-1}$ is an isomorphism of $\Ol_S$-modules.
	\end{itemize}
For $s\in S$, consider the function \[\tau_s: \Z\ra \Z_{\geq 0}, \quad m\mapsto \dim_{\kappa(s)}\gr^m_C(M_s)=\dim_{\kappa(s)}C_s^m/C_s^{m+1}.\] As $\gr_C^mM$ are locally free, the function $s\mapsto \tau_s$ is locally constant on $S$, which takes values in the set of maps $\Z\ra \Z_{\geq 0}$ with finite support. We call this function the type of $(M, C^\bullet, D_\bullet, \varphi)$.
Let $F\tr{-Zip}(S)$ be the category of $F$-zips over $S$. 
Letting $S$ vary, we get a smooth algebraic stack $F\tr{-Zip}$, which decomposes as
\[ F\tr{-Zip}=\coprod_\tau F\tr{-Zip}^\tau,\]
where $\tau$ runs through the set of functions $\Z\ra \Z_{\geq 0}$ with finite support, and each $F\tr{-zip}^\tau$ is the open and closed substack classifying $F$-zips of type $\tau$, cf. \cite{MW} Proposition 2.2.

We have also the notion of $F$-zips with additional structure. Let $G/\F_p$ be a connected reductive group, $\mu$ a cocharacter of $G$ defined over a finite field $\kappa|\mathbb{F}_p$.
Let $P_+$ (resp. $P_-$) be the parabolic subgroup of $G_\kappa$ such that its Lie algebra is the sum of spaces with non-negative weights (resp. non-positive weights) in $\Lie(G_\kappa)$.
We will also write $U_+$ (reps. $U_-$) for the unipotent radical of $P_+$ (resp. $P_-$).
Let $L$ be the common Levi subgroup of $P_+$ and $P_-$.
Now assume that $S$ is a scheme over $\kappa$. A $G$-zip of type $\mu$ over $S$ is given by (cf. \cite{PWZ15} Theorem 7.13)
	\begin{enumerate}
		\item either a tuple $\underline{I} = (I,I_+,I_-,\varphi)$ consisting of
		\begin{itemize}
			\item a right $G$-torsor $I$ over $S$,
			\item a right $P_+$-torsor $I_+ \subset I$,
			\item a right $P_-^{(p)}$-torsor $I_- \subset I$,
			\item an isomorphism of $L^{(p)}$-torsors $\varphi: I_{+}^{(p)} / U_+^{(p)} \to I_{-}/ U_-^{(p)} $.
		\end{itemize}
		\item or an exact $\F_p$-linear tensor functor 
		$\mathfrak{z}: \Rep\,G\ra F\tr{-Zip}(S)$ of type $\mu$, which means that the graded fiber functor $\gr_C\circ \mathfrak{z}$ and $\gamma_\mu$ are fpqc-locally isomorphic. Here $\gr_C: F\tr{-Zip}(S)\ra \mathrm{GrVect}(S)$ is the functor taking grading with respect to the filtration $C^\bullet$ and
		$\gamma_\mu: \Rep\,G\ra \mathrm{GrVect}(\kappa)$ is the tensor functor defined by $\mu$.
	\end{enumerate}
One can check that the category of $F$-zips of rank $n$ is equivalent to the category of $\GL_n$-zips, see \cite{PWZ15}.

The category of $G$-zips of type $\mu$ over $S$ will be denoted by $G\tr{-Zip}^\mu(S)$.
This defines a category fibered in groupoids $G\tr{-Zip}^\mu$ over $\kappa$.
\begin{theorem}[\cite{PWZ15}]
	The fibered category $G\tr{-Zip}^\mu$  is a smooth algebraic stack of dimension 0 over $\kappa$, which has the form of a quotient stack $[E_{G,\mu} \backslash G_\kappa] $, where $E_{G,\mu} $ is certain algebraic group acting on $G_\kappa$.
\end{theorem}
Let $B \subset G$ be a Borel subgroup and $T \subset B$ a maximal torus.
Let $W = W(B,T)$ be the absolute Weyl group, and $\Delta$ the set of simple roots defined by $B$.
Let $J \subset \Delta$ be the simple roots correspond to $P_+$.
Let $W_J$ be the subgroup of $W$ generated by $J$, and $^J W$ the set of elements $w$ such that $w$ is the element of minimal length in some coset $W_J w'$. By \cite{PWZ11} section 6, there is a partial order $\preceq$ on $^J W$, and we have a homeomorphism of topological spaces \[|[E_{G,\mu} \backslash G_\kappa]|\simeq (^J W, \preceq).\]

\section{De Rham $F$-gauges}

Let $X/k$ be a smooth scheme as in the last section.
We introduce and study the notion of de Rham $F$-gauges on $X$, inspired by \cite{MW, Kat72, OV} and \cite{Bhatt}. The notion of $F$-gauges was originally introduced by Fontaine-Jannsen in \cite{FJ}. Recently Drinfeld \cite{Dri20} and Bhatt-Lurie \cite{BL1, BL2} have extended this notion to prismatic cohomology for $p$-adic formal schemes in mixed characteristic. Here we restrict to characteristic $p$ and consider only de Rham cohomology.

\subsection{Explicit definition}\label{subsection explicit} Let $X$ be as above. Roughly a de Rham $F$-gauge over $X$ is an $F$-zip enriched by a flat connection.
Recall $Fr_X: X\ra X$ is the absolute Frobenius map.
\begin{definition}\label{def F-gauge}
\begin{enumerate}
\item
A de Rham $F$-gauge (in vector bundle) over $X$ is a tuple $\ul{\E}=(\E, \nabla, C^\bullet, D_\bullet, \varphi)$, where
\begin{itemize}
	\item $(\E, \nabla)$ is a flat vector bundle on $X$,
	\item $C^\bullet$ is a descending filtration on $\E$, which satisfies the Griffiths transversality condition with respective to $\nabla$: for any $i$, we have $\nabla(C^i)\subset C^{i-1}\otimes\Omega_X^1$.
	\item  $D_\bullet$ is an ascending filtration on $\E$, which is horizontal with respective to $\nabla$, such that the induced connection on $\gr_{D}\E$ has zero $p$-curvature,
	\item $\varphi$ is an $\Ol_X$-linear isomorphism
	\[\varphi: Fr_X^\ast(\gr_{C}\E, \theta)\st{\sim}{\lra} (\gr_{D}\E, \psi), \]
	where $\theta=\gr_{C}\nabla: \gr_C\E\ra \gr_C\E\otimes\Omega_X^1$ is the graded Higgs field as before, and \[\psi: \gr_{D}\E\ra \gr_{D}\E\otimes Fr_X^\ast\Omega_X^1\] is the $\Ol_X$-linear morphism induced by $\nabla$ and the zero $p$-curvature condition. 
\end{itemize}
\item A morphism of de Rham $F$-gauges $f: \ul{\E_1}\ra \ul{\E_2}$ is a morphism of flat vector bundles $f: (\E_1,\nabla_1)\ra (\E_2,\nabla_2)$, which is compatible with the filtrations $C_1^\bullet, C_2^\bullet$ and $D_{1,\bullet}, D_{2,\bullet}$, such that on graded vector bundles we have $\gr_D(f)\circ \varphi_1\simeq \varphi_2\circ Fr_X^\ast(\gr_C(f))$.
\end{enumerate}
We denote the category of de Rham $F$-gauges on $X$ by $F\tr{-Gauge}_{dR}(X)$.
\end{definition}
We want to emphasis that the graded Higgs field $\theta=\gr_{C}\nabla$ has the form 
\[\theta=\bigoplus_i\theta_i,\quad \tr{and}\quad \theta_i: \gr_C^i\E=C^i/C^{i+1}\ra \gr_C^{i-1}\E\otimes\Omega_X^1=(C^{i-1}/C^{i-2})\otimes\Omega_X^1.\]
It can be viewed as a section \[\theta\in H^0\Big(X, \bigoplus_i\H om(\gr_C^i\E, \gr_C^{i-1}\E)\otimes\Omega_X^1\Big)\subset H^0\Big(X, \E nd(\gr_C\E)\otimes\Omega_X^1\Big).\]
Similarly, the morphism $\psi$ induced by the $p$-curvature condition of the graded connection $\gr_{D}\nabla$ has the form
\[\psi=\bigoplus_i\psi_i,\quad \tr{and}\quad \psi_i: D_i/D_{i-1}\ra (D_{i-1}/D_{i-2})\otimes Fr_X^\ast\Omega_X^1,\]  which can be viewed as a section
\[\psi\in  H^0\Big(X, \bigoplus_i\H om(\gr_D^i\E, \gr_D^{i-1}\E)\otimes Fr_X^\ast\Omega_X^1\Big)\subset H^0\Big(X, \E nd(\gr_D\E)\otimes Fr_X^\ast\Omega_X^1\Big).\]
By projection formula and using $\gr_D\E\simeq Fr_X^\ast\gr_C\E$, we have
\[ \begin{split}
&H^0\Big(X, \bigoplus_i\H om(\gr_D^i\E, \gr_D^{i-1}\E)\otimes Fr_X^\ast\Omega_X^1\Big)\\=& H^0\Big(X, \bigoplus_iFr_{X\ast}\H om(\gr_D^i\E, \gr_D^{i-1}\E)\otimes \Omega_X^1\Big)\\ =& H^0\Big(X, \bigoplus_iFr_{X\ast}\H om(Fr_X^\ast\gr_C^i\E, Fr_X^\ast\gr_C^{i-1}\E)\otimes \Omega_X^1\Big).
\end{split}\] Applying basic properties of $p$-curvature (cf. \cite{Kat70} Proposition 5.2), we have in fact
\[ \begin{split}
\psi \in & H^0\Big(X, \big(\bigoplus_iFr_{X\ast}\H om(Fr_X^\ast\gr_C^i\E, Fr_X^\ast\gr_C^{i-1}\E)\otimes \Omega_X^1\big)^{\nabla^{can}}\Big)\\=& H^0\Big(X, \bigoplus_i\H om(\gr_C^i\E, \gr_C^{i-1}\E)\otimes\Omega_X^1\Big).
\end{split}\]
The isomorphism $\varphi$ in the last term of the above definition is required to preserve the grading structures on both sides. It is equivalent to the corresponding bijective $p$-linear morphism $\gr_C\E\ra\gr_D\E$, which makes the following digram commute:
\[\xymatrix{\gr_C\E\ar[rr]^-\theta\ar[d]& & \gr_C\E\otimes\Omega_X^1\ar[d]\\
	\gr_D\E\ar[rr]^-\psi& &\gr_D\E\otimes Fr^\ast_X \Omega_X^1,
	} \]where the right vertical map is the tensor product of the $p$-linear morphism $\gr_C\E\ra\gr_D\E$ with the natural $p$-linear  map $\Omega_X^1\ra Fr^\ast_X \Omega_X^1$.
We also remark that there is an obvious generalization of the above definition by considering quasi-coherent sheaves instead of  vector bundles.

\begin{example}\label{example F-gauge}
Let $f: Y\ra X$ be proper smooth scheme over $X$ of relative dimension $n$. Assume that the Hodge-de Rham spectral sequence \[E_1^{a,b}=R^bf_\ast(\Omega_{Y/X}^a)\Longrightarrow R^{a+b}f_\ast(\Omega_{Y/X}^\bullet)\] degenerates at $E_1$, and all the Hodge cohomology sheaves $R^bf_\ast(\Omega_{Y/X}^a)$ are locally free of finite rank over $X$. Then for each $0\leq i\leq 2n$, the relative de Rham cohomology \[H^i_{\dR}(Y/X)=R^if_\ast(\Omega_{Y/X}^\bullet)\] admits a canonical de Rham $F$-gauge structure: $\nabla$ is given by the Gauss-Manin connection, $C^\bullet$ is given by the Hodge filtration, $D_\bullet$ is given by the conjugate filtration (induced from the conjugate spectral sequence which also degenerates), and the isomorphism \[\varphi:  Fr_X^\ast(\gr_CH^i_{\dR}(Y/X),\theta)\st{\sim}{\lra}(\gr_DH^i_{\dR}(Y/X),\psi)\] is given by Katz in \cite{Kat72} Theorem 3.2. 

The above assumption on $Y\ra X$ holds in the following cases: abelian schemes, proper smooth curves, K3 surfaces, smooth complete intersection in a projective space $\mathbb{P}_{X}^n$ over $X$, etc. See \cite{MW} 7.4.
\end{example}

The usual operations on flat vector bundles extend naturally to de Rham $F$-gauges: we have direct sums, dual, tensor products etc, which make $F\tr{-Gauge}_{dR}(X)$ a tensor exact category. If considering quasi-coherent de Rham $F$-gauges, we would get a tensor abelian category. As an example, we discuss the tensor product $\ul{\E}=\ul{\E_1}\otimes\ul{\E_2}$ of two de Rham $F$-gauges $\ul{\E_1}=(\E_1,\nabla_1, C_1^\bullet, D_{1,\bullet}, \varphi_1)$ and $\ul{\E_2}=(\E_2,\nabla_2, C_2^\bullet, D_{2,\bullet}, \varphi_2)$:
\begin{itemize}
	\item $(\E, \nabla)=(\E_1,\nabla_1)\otimes (\E_2, \nabla_2)$ is the usual tensor product of two flat vector bundles (cf. \cite{Kat70} 1.1.1),
	\item $C^\bullet=C^\bullet_1\otimes C^\bullet_2$ is the usual tensor product of two decreasing filtrations on $\E_1\otimes \E_2$,
	\item $D_\bullet=D_{1,\bullet}\otimes D_{2,\bullet}$ is the tensor product of two increasing filtrations on $\E_1\otimes \E_2$,
	\item For each $a\in \Z$, $\varphi_a: Fr_X^\ast\gr_C^a(\E_1\otimes\E_2)\ra \gr_D^a(\E_1\otimes\E_2)$ is the isomorphism which makes the following diagram commute:
	\[\xymatrix{ Fr_X^\ast\gr_C^a(\E_1\otimes\E_2)\ar[d]^{\varphi_a}\ar[rr]^-\sim & & \bigoplus_iFr_X^\ast \gr_C^i\E_1\otimes Fr_X^\ast \gr_C^{a-i}\E_2\ar[d]^{\bigoplus_i(\varphi_{1,i}\otimes\varphi_{2,a-i})}\\
		\gr_D^a(\E_1\otimes\E_2)\ar[rr]^-\sim & &\bigoplus_i\gr_D^i\E_1\otimes\gr_D^{a-i}\E_2.
		} \]
		The induced morphisms $\theta_a=\bigoplus_i(\theta_{1,i}\otimes1+1\otimes\theta_{2,a-i})$ and $\psi_a=\bigoplus_i(\psi_{1,i}\otimes1+1\otimes\psi_{2,a-i})$ are matched under $\varphi_a$ by the compatibilities of corresponding components.
\end{itemize}

By definition, we have a natural forgetful functor
\[ F\tr{-Gauge}_{dR}(X)\lra F\tr{-Zip}(X),\]which is an equivalence of categories if $X=\Spec\,k$ with $k|\F_p$ is a field. In particular, we can define the notion of de Rham $F$-gauges of type $\tau$ associated to a locally constant functor $\tau$ as in subsection \ref{subsection F-zip}, and by definition we have the forgetful functor $F\tr{-Gauge}_{dR}^\tau(X)\lra F\tr{-Zip}^\tau(X)$.

On the other hand, we can construct de Rham $F$-gauges from $p$-torsion Fontaine modules (the construction is implicitly contained in \cite{OV} section 4.6).
\begin{proposition}\label{prop MF dR gauge}
	Let $\X_2/W_2$ be a lifting of $X/k$. Then
there is a natural functor
\[ \MF^\nabla_{[0,p-2]}(\X_2/W_2) \lra F\tr{-Gauge}_{dR}(X),\]
which is fully faithful, with essential image the de Rham $F$-gauges $\ul{\E}$ of nilpotent level $\leq p-1$, i.e. $\psi_\nabla^{p-1}=0$. 
\end{proposition}
\begin{proof}
For an object $(M,\nabla, \Fil, \Phi)\in \MF^\nabla_{[0,p-2]}(\X_2/W_2)$, set $C^\bullet=\Fil$. We construct a conjugate filtration
$D_\bullet$ on $M$  as follows (see also \cite{OV} the paragraph above Theorem 4.17).
For each $i$, set
\[D_i(M)=\Phi(Fr_X^\ast(\bigoplus_{j\leq i}\gr_C^jM))\subset M.\]
Then $D_\bullet$ is an increasing filtration on $M$ and by construction
we have an isomorphism $\varphi_i: Fr_X^\ast\gr_{C}^iM\simeq \gr_{D}^iM$ for each $i$. By construction, $D_\bullet$ is horizontal with respect to $\nabla$ and the associated graded connection on $\gr_{D}M$ has zero $p$-curvature. 
By the last condition in the definition of the $p$-torsion Fontaine module $(M,\nabla, \Fil, \Phi)$, we have $\varphi: Fr_X^\ast\theta_\nabla \simeq \psi_\nabla$. Then $(M,\nabla, \Fil, \Phi)\mapsto (M,\nabla,\Fil, D_\bullet,\varphi)$ defines a functor $\MF^\nabla_{[0,p-2]}(\X_2/W_2) \lra F\tr{-Gauge}_{dR}(X)$.

By Proposition \ref{prop p-torsion FM},
this functor is fully faithful, with essential image the de Rham $F$-gauges $\ul{\E}$ of nilpotent level $\leq p-1$.
\end{proof}

\begin{remark}
Contrary to the notion of $p$-torsion Fontaine modules, it is important to note that the notion of de Rham $F$-gauges does not rely on any lifting of $X$.
\end{remark}

Let $n=\dim\,X$. 
For a de Rham $F$-gauge  $\ul{\E}=(\E, \nabla, C^\bullet, D_\bullet, \varphi)$ over $X$, we can form the following filtered complexes. First, as usual we can view the de Rham complex $DR(\E,\nabla)$ as a complex of $\D_X$-modules (see \cite{FC} Chapter VI section 3). Then we have:
\begin{itemize}
	\item The $C^\bullet$-filtered de Rham complex $DR(\E,\nabla)=(\E\otimes \Omega^\bullet_X, \nabla)$, with
	\[ C^a(DR(\E,\nabla))=C^{a-\bullet}(\E)\otimes\Omega_X^\bullet,\quad a\in\Z. \]
	This filtration is compatible with the order filtration $F_\bullet$ of $\D_X$.
	The associated graded complex is: for any $a\in\Z$,
	\[\gr_C^aDR(\E,\nabla)=\Big[0\ra \gr^a_C\E\ra \gr^{a-1}_C\E\otimes\Omega_X^1\ra\cdots \ra \gr^{a-n}_C\E\otimes\Omega_X^n\ra 0\Big].\]
	In particular, we get an object  \[\gr_CDR(\E,\nabla)\in D\Big(Mod^\bullet(\gr_F\D_X)\Big)\simeq D\Big(Mod^\bullet(\Sym\T_X)\Big).\]
	\item The $D_\bullet$-filtered de Rham complex $DR(\E,\nabla)=(\E\otimes \Omega^\bullet_X, \nabla)$, where the filtration is induced from the $D_\bullet$-filtration on $\E$ as follows. Since the induced $D_\bullet$-filtration on each term $\E\otimes\Omega_X^i$ is $\mathcal{J}_X$-compatible with the conjugate filtration $N^\bullet$ on $\D_X$ (cf. subsection \ref{subsection conj fil}), we can view $DR(\E,\nabla)$ together with the $D_\bullet$-filtration as an object in $DF(\D_X, \mathcal{J}_X)$ (cf. \cite{OV} Definition 3.13). The associated graded complex defines an object  \[\gr_{D} DR(\E,\nabla)\in D\Big(Mod^\bullet(\gr_N\D_X)\Big).\]
	\item The $D_\bullet$-filtered $p$-curvature complex $K(\E,\psi)=(\E\otimes Fr_X^\ast\Omega^\bullet_X,\psi)$, with
	\[
	 D_a(K(\E,\psi))=D_{a-\bullet}(\E)\otimes Fr_X^\ast\Omega_X^\bullet,\quad a\in\Z. \]
	 The associated graded complex is: for any $a\in\Z$,
	 	\[\gr_D^a(K(\E,\psi))=\Big[0\ra \gr^a_D\E\ra \gr^{a-1}_D\E\otimes Fr_X^\ast\Omega_X^1\ra\cdots \ra \gr^{a-n}_D\E\otimes Fr_X^\ast\Omega_X^n\ra 0\Big],\]
	 which	we can view as a graded complex of $Fr_X^\ast\Sym\T_X$-modules (cf. \cite{FC} Chapter VI section 3):
	 	it defines an object  \[\gr_DK(\E,\psi)\in D\Big(Mod^\bullet(Fr_X^\ast\Sym\T_X)\Big).\]
	 
\end{itemize}
Recall that we have an equivalence of categories (cf. subsection \ref{subsection conj fil}) \[C_X: D\Big(Mod^\bullet(\gr_N\D_X)\Big)\ra D\Big(Mod^\bullet(Fr_X^\ast\Sym\T_X)\Big).\]
\begin{proposition}\label{prop Frob gauge complexes}
We have a quasi-isomorphism for graded complexes 
\[\varphi: Fr_X^\ast\gr_{C}DR(\E,\nabla) \st{\sim}{\lra} C_X(\gr_{D} DR(\E,\nabla)), \]
which extends the isomorphism $\varphi$ in the definition of the de Rham $F$-gauge $\ul{\E}$.
\end{proposition}
\begin{proof}
As $\ul{\E}$ satisfies the conditions in Definition \ref{def F-gauge}, by the above description, we see $\varphi$ induces an isomorphism between each term of the complexes $Fr_X^\ast\gr_{C}DR(\E,\nabla) $ and $\gr_{D} K(\E,\psi)$. To show it also induces an isomorphism between differentials, it suffices to see for any $a\in \Z$, the first differentials $\theta_a: \gr_C^a\E\ra \gr_C^{a-1}\E\otimes\Omega_X^1$ and $\psi_a: \gr_D^a\E\ra \gr_{D}^{a-1}\E\otimes Fr_X^\ast\Omega_X^1$ are matched. But this also follows from the definition of de Rham $F$-gauges.

On the other hand, we have the equivalence of categories
\[C_X^{-1}: D\Big(Mod^\bullet(Fr_X^\ast\Sym\T_X)\Big)\st{\sim}{\lra} D\Big(Mod^\bullet(\gr_N\D_X)\Big).\] One checks by construction that \[C_X^{-1}(\gr_DK(\E,\psi))\simeq \gr_{D} DR(\E,\nabla).\] Therefore, we get a quasi-isomorphism for graded complexes 
\[\varphi: Fr_X^\ast\gr_{C}DR(\E,\nabla) \st{\sim}{\lra} C_X(\gr_{D} DR(\E,\nabla)).\]
\end{proof}

The isomorphism in Proposition \ref{prop Frob gauge complexes} can be rewritten as
\[\varphi: \pi_X^\ast \gr_{C}DR(\E,\nabla) \st{\sim}{\lra} C_X(\gr_{D} F_{X/k\ast}DR(\E,\nabla)).\]
For the de Rham $F$-gauge $\ul{\E}=(\E, \nabla, C^\bullet, D_\bullet, \varphi)$ over $X$, from the $C^\bullet$-filtered complex $DR(\E,\nabla)$ we get the Hodge-de Rham spectral sequence
\['E_1^{a,b}=H^{a+b}\Big(X, \gr_C^aDR(\E,\nabla)\Big)\Longrightarrow H^{a+b}_{\dR}(X/k, (\E,\nabla)). \]
From the $D_\bullet$-filtered complex $F_{X/k\ast}DR(\E,\nabla)$, we get the conjugate spectral sequence
\[''E_1^{a,b}=H^{a+b}\Big(X', \gr_{D}^aF_{X/k\ast}DR(\E,\nabla)\Big)\Longrightarrow H^{a+b}_{\dR}(X/k, (\E,\nabla)).\]
The Cartier-twisted Frobenius \[\varphi: \pi_X^\ast\gr_{C}DR(\E,\nabla) \st{\sim}{\lra} C_X(\gr_{D}F_{X/k\ast}DR(\E,\nabla))\] induces an isomorphism
\[\varphi: Fr_k^\ast('E_1^{a,b})=Fr_k^\ast H^{a+b}\Big(X,\gr_C^aDR(\E,\nabla)\Big)\st{\sim}{\lra} H^{a+b}\Big(X', \gr_{D}^aF_{X/k\ast}DR(\E,\nabla)\Big)=''E_1^{a,b}.\]
Therefore, by the same arguments as in \cite{Kat70} Proposition 2.3.2 and \cite{MW} subsection 7.5, we get
\begin{proposition}\label{prop coh F-gauge}
Let $X$ be a proper and smooth variety over $k$. Assume that the Hodge-de Rham spectral sequence \['E_1^{a,b}=H^{a+b}\Big(X, \gr_C^aDR(\E,\nabla)\Big)\Longrightarrow H^{a+b}_{\dR}(X/k, (\E,\nabla)) \]
degenerates. 
Then the conjugate spectral sequence also degenerates. Moreover, for each $0\leq i\leq 2n=2\dim\,X$, we get a de Rham  $F$-gauge (= $F$-zip) on $k$ given by
\[H^i(X,\ul{\E}):= \Big(H^i_{\dR}(X/k, (\E,\nabla)), C^\bullet, D_\bullet, \varphi\Big),\] where $C^\bullet$ (resp. $D_\bullet$) is the induced Hodge filtration (resp. conjugate filtration) on $H^i_{\dR}(X/k,(\E,\nabla))$, and for each $0\leq a\leq i$, $\varphi_a$ is the isomorphism induced from the Frobenius on coefficients, which sits in the following commutative diagram
\[\xymatrix{ Fr_k^\ast H^{i}\Big(X, \gr_C^aDR(\E,\nabla)\Big)\ar[rr]^\sim \ar[d]^\sim & & H^{i}(X, \gr_{D}^aK(\E,\psi))\ar[d]^\sim \\
Fr_k^\ast\gr_C^aH^i_{\dR}(X/k,(\E,\nabla))\ar[rr]^{\varphi_a}& & \gr_D^aH^i_{\dR}(X/k,(\E,\nabla)).
}\]
\end{proposition}

One can obviously generalize the above proposition to the relative setting for a proper smooth morphism $Y\ra X$ (as in Example \ref{example F-gauge}) with coefficients in a de Rham $F$-gauge on $Y$.  We leave this task to the reader.
\begin{remark}\label{rmk log F-gauge}
Let $\ov{X}$ be a proper smooth scheme over $k$ and $X\subset \ov{X}$ an open subscheme such that $D=\ov{X}\setminus X$ is a normal crossing divisor of $\ov{X}$. One has the classical notion of flat connections over $\ov{X}$ with log poles at $D$. In fact, we have also natural extensions of all the discussions in this subsection to this logarithmic setting. See \cite{MW} subsection 7.7 for some discussions of the extension of $F$-zips to log $F$-zips.
\end{remark}

\subsection{De Rham $F$-gauges via syntomification}

	Let $X$ be a $p$-adic formal scheme over $\Z_p$. By Drinfeld \cite{Dri18, Dri20} and Bhatt-Lurie \cite{BL1, BL2, Bhatt},
	attached to $X$, one has 3 stacks (with increased complicities): $X^{\Prism}, X^{\mathcal{N}}$ and $X^{\tr{syn}}$. The (stable $\infty$-) category of prismatic $F$-gauges is defined by
	\[\mathcal{D}_{qc}(X^{\tr{syn}}),\]which contains prismatic $F$-gauges  in perfect complexes and vector bundles respectively
	\[\tr{Perf}(X^{\tr{syn}}), \quad \tr{Vect}(X^{\tr{syn}}).\]
	Specializing to our smooth scheme $X/k$ as above, the stacks $X^{\Prism}$ and $X^{\mathcal{N}}$ are stacks over $W=W(k)$, while $X^{\tr{syn}}$ is an stack over $\Z_p$, obtained by gluing the two copies of $X^{\Prism}$ inside  $X^{\mathcal{N}}$. In particular, we can consider the special fiber \[\ov{X^{\tr{syn}}}:=X^{\tr{syn}}\otimes \F_p\] of $X^{\tr{syn}}$. The constructions are functorial, so that we get a morphism of stacks $\ov{X^{\tr{syn}}}\ra \ov{k^{\tr{syn}}}$. In \cite{Bhatt} chapter 2, this is denoted by $X^{\ov{C}}\ra\ov{C}$. We will consider the category $\Vect(\ov{X^{\tr{syn}}})$ of vector bundles on $\ov{X^{\tr{syn}}}$. Note that it can be viewed as a full subcategory $\Vect(\ov{X^{\tr{syn}}}) \hookrightarrow \tr{Perf}(X^{\tr{syn}})$.
	The following theorem is implicitly contained in \cite{Bhatt}.
	\begin{theorem}\label{thm F-gauge via stack}
		Suppose that $X/k$ is a smooth algebraic variety over a perfect field $k$ of characteristic $p$. We have an equivalence of categories
	\[F\tr{-Gauge}_{dR}(X)\simeq \Vect(\ov{X^{\tr{syn}}}).\]
\end{theorem}
\begin{proof}
The proof is rather standard from the stacky approach as in \cite{Bhatt} and \cite{Dri20}, so we just sketch it here.
We need to understand the geometry of the stack $\ov{X^{\tr{syn}}}$. This will proceed step by step.
\begin{itemize}
	\item[Step 1.] First, consider the de Rham stack $(X/k)^{dR}$. We have a flat cover \[X\ra (X/k)^{dR},\] and its Cech nerve can be identified with the simplicial formal scheme $P_X(\bullet)$ defined by taking  PD formal completion of the Cech nerve of $X\ra \Spec\,k$ along the diagonal copy of $X$).
	 Then $\Vect((X/k)^{dR})$ is equivalent to the flat vector bundles on $X$. This is implied by \cite{Dri18} Theorem 2.4.2 or \cite{BL2} Theorem 6.5.
	 
	\item[Step 2.] Consider the Hodge filtered de Rham stack $(X/k)^{dR,+}$. By construction, we have a morphism $(X/k)^{dR,+}\ra \A^1/\mathbb{G}_m$ and a flat cover \[X\times \A^1/\mathbb{G}_m\ra (X/k)^{dR,+}\] over $\A^1/\mathbb{G}_m$.
	Moreover, we have an open immersion $(X/k)^{dR}\hookrightarrow (X/k)^{dR,+}$ corresponding to $\mathbb{G}_m/\mathbb{G}_m\hookrightarrow \A^1/\mathbb{G}_m$. Its complement is denoted by $(X/k)^{Higgs}$. Vector bundles on $(X/k)^{Higgs}$ is equivalent to graded Higgs bundles on $X$ such that the Higgs field decreases degree 1 and is nilpotent. The category $\Vect((X/k)^{dR,+})$ is equivalent to the category of triples $(\E, \nabla, \Fil^\bullet)$, where $(\E, \nabla)$ is a flat vector bundle on $X$, $\Fil^\bullet$ is a decreasing filtration on $\E$ which satisfies the Griffiths transversality condition with respect to $\nabla$, cf. \cite{Bhatt} Remark 2.5.8.
	
	\item[Step 3.] Consider the conjugate filtered de Rham stack $(X/k)^{dR,c}$. By construction, we have a morphism $(X/k)^{dR,c}\ra \A^1/\mathbb{G}_m$ and a flat cover \[ (X/k)^{dR,c}\ra X'\times \A^1/\mathbb{G}_m\] over $\A^1/\mathbb{G}_m$. Here $X'$ is the Frobenius pullback of $X$ and the relative Frobenius morphism $F_{X/k}: X\ra X'$ factors through $(X/k)^{dR}$, cf. \cite{Bhatt} Remark 2.7.4. We have an open immersion $(X/k)^{dR}\hookrightarrow (X/k)^{dR,c}$ corresponding to $\mathbb{G}_m/\mathbb{G}_m\hookrightarrow \A^1/\mathbb{G}_m$. Its complement is denoted by $(X/k)^{F-Higgs}$. Vector bundles on $(X/k)^{F-Higgs}$ is equivalent to graded $F$-Higgs bundles on $X$ such that the Higgs field decreases degree 1 and is nilpotent. Vector bundles on $(X/k)^{dR,c}$ can be described as triples $(\E,\nabla, \Fil_\bullet)$, where $(\E,\nabla)$ is a flat vector bundle on $X$ with nilpotent $p$-curvature, $\Fil_\bullet$ is an increasing filtration on $\E$ which is horizontal with respect to $\nabla$, such that the graded connection on $\gr_{\Fil_\bullet}\E$ has zero $p$-curvature.
	
	\item[Step 4.] Consider the stack $\ov{X^{\mathcal{N}} }:=X^{\mathcal{N}}\otimes_W k$, which is denoted by $X^C$ in \cite{Bhatt} Definition 2.8.4. There are open immersions \[(X/k)^{dR,+}\ra \ov{X^{\mathcal{N}}} \quad \tr{and}\quad (X/k)^{dR,c}\ra \ov{X^{\mathcal{N}}},\] with image $\ov{X^{\mathcal{N}}}_{u=0}$ and $\ov{X^{\mathcal{N}}}_{t=0}$ respectively. The stack $\ov{X^{\mathcal{N}}}$ lives over $\ov{k^{\mathcal{N}}}=k^{\mathcal{N}}\otimes_Wk=(\Spf\,k[u,t]/(ut))/\mathbb{G}_m$, so that there are coordinates $u$ and $t$. Vector bundles on $\ov{X^{\mathcal{N}}}$ can be described as \[(\E_1,\nabla_1,  \Fil^\bullet, \E_2,\nabla_2, \Fil_\bullet, \varphi)\] such that $(\E_1,\nabla_1,  \Fil^\bullet)\in \Vect((X/k)^{dR,+})$, $(\E_2,\nabla_2,  \Fil_\bullet)\in \Vect((X/k)^{dR,c})$, and $\varphi: Fr_X^\ast\gr_{\Fil^\bullet}\E_1\simeq \gr_{\Fil_\bullet}\E_2$ is an isomorphism of graded vector bundles.
	
	\item[Step 5.] Finally, by construction we have a pushout diagram
	\[\xymatrix{ 
		(X/k)^{dR}\coprod (X/k)^{dR}\ar[r]\ar[d]& \ov{X^{\mathcal{N}}} \ar[d]\\
		(X/k)^{dR}\ar[r] & \ov{X^{\tr{syn}}}.
		}\]
	In other words, $\ov{X^{\tr{syn}}}$ is obtained by gluing the two copies of $(X/k)^{dR}$ inside $\ov{X^{\mathcal{N}}}$.  From the above description of $\Vect(\ov{X^{\mathcal{N}}})$ we get that $\Vect(\ov{X^{\tr{syn}}})$  is equivalent to the category of tuples $(\E, \nabla, \Fil^\bullet, \Fil_\bullet, \varphi)$ which satisfy the condition as in Definition \ref{def F-gauge}.
\end{itemize}

\end{proof}

\subsection{De Rham $F$-gauges with $G$-structure}
We want a notion of de Rham $F$-gauges with $G$-structure, as a strengthen of the notion of $G$-zips in subsection \ref{subsection F-zip}. So recall our notations there.	Let $G/\F_p$ be a connected reductive group, $\mu: \G_m\ra G$ a cocharacter over a finite extension $\kappa|\F_p$ with associated conjugacy class $\{\mu\}$, $P_+=P_\mu$ the associated parabolic, $P_{-}$ the opposite parabolic, $L$ the common Levi subgroup of $P_+$ and $P_{-}$. 

Let $X/\kappa$ be a smooth scheme and $I$ a $G$-bundle on $X$. A flat connection $\nabla$ on $I$ is given by an isomorphism \[\nabla: pr_2^\ast I\st{\sim}{\lra}pr_1^\ast I\] such that $\delta^\ast\nabla=\tr{Id}_I$ and $pr_{13}^\ast\nabla=pr_{12}^\ast\nabla\circ pr_{23}^\ast\nabla$. Here are the meaning of the notations: let $\Delta^2(1)$ be the first order neighborhood of the diagonal $X\hookrightarrow X\times_\kappa X$, $\delta: X\hookrightarrow \Delta^2(1)$ the induced embedding and $pr_i: \Delta^2(1)\ra X$ the induced projections for $i=1,2$. Let  $\Delta^3(1)$ be the first order neighborhood of the diagonal $X\hookrightarrow X\times_\kappa X\times_\kappa X$, with induced projections $pr_{ij}: \Delta^3(1)\ra \Delta^2(1)$ for $i, j \in \{1,2,3\}$. It is well known that the datum of a flat $G$-bundle $(I,\nabla)$ is equivalent to an exact tensor functor \[\Rep\,G\ra \Vect^\nabla(X),\] where $\Vect^\nabla(X)$ is the category of flat vector bundles on $X$, cf. \cite{Lov} subsection 2.3.
	\begin{definition}\label{def G-gauge}
	A de Rham $F$-gauge with $(G,\mu)$-structure over $X$ is given by
	a tuple \[\underline{I} = (I, \nabla, I_+,I_-,\varphi),\] consisting of
	\begin{itemize}
		\item a  $G$-bundle $I$ over $X$ with a flat connection $\nabla$,
		\item a $P_+$-bundle $I_+ \subset I$,
		\item a  $P_-^{(p)}$-bundle $I_- \subset I$ on which $\nabla$ induces a connection, which has zero $p$-curvature on the graded  $L^{(p)}$-bundle,
		\item an isomorphism of $L^{(p)}$-bundles \[\varphi: I_{+}^{(p)} / U_+^{(p)} \st{\sim}{\lra} I_{-} / U_-^{(p)} ,\] under which the associated Higgs field and $p$-curvature coincides  \[\varphi: Fr_X^\ast\theta_\nabla\simeq \psi_\nabla.\]
	\end{itemize}
	Sometimes we also call it a $(G,\mu)$-de Rham $F$-gauge.
	Denote the category of such objects by $F\tr{-Gauge}_{dR}^{G,\mu}(X)$.
\end{definition}
Let us explain the last term following \cite{Dri23} (but note that our sign $\pm$ correspond to the $\mp$ of loc. cit.). Let $\u_{\pm}=\Lie\,U_{\pm}$. Then $\u_-$ (resp. $\u_-^{(p)}$) forms a representation of $L$ (resp. $L^{(p)}$), thus we can view it as a representation of
$P=P_+$ (resp. $P_-^{(p)}$) by the quotient $P\ra L$ (resp. $P_-^{(p)}\ra L^{(p)}$). For $(I, \nabla, I_+,I_-)$ as above, we get the associated vector bundles \[I_{+}(\u_-)=I_+\times^{P_+}\u_- \quad \tr{and} \quad I_{-}(\u_-^{(p)})=I_-\times^{P_-^{(p)}}\u_-^{(p)}.\] Then the Higgs field attached to $(\nabla, I_+)$ is given by a map
\[\theta_\nabla: \T_X\lra I_{+}(\u_-), \]or equivalently a global section
$\theta_\nabla\in H^0(X, I_{+}(\u_-)\otimes\Omega_X^1)$. The $p$-curvature attached to $(\nabla, I_-)$ is given by a map
\[\psi_\nabla: \T_X\lra I_{-}(\u_-^{(p)}),\]or equivalently a global section $\psi_\nabla\in H^0(X, I_{-}(\u_-^{(p)})\otimes \Omega_X^1)$. From the isomorphism of $L^{(p)}$-bundles  \[\varphi: I_{+}^{(p)} / U_+^{(p)} \st{\sim}{\lra} I_{-} / U_-^{(p)},\] we get an induced isomorphism \[\varphi: Fr_X^\ast I_{+}(\u_-)=I_{+}(\u_-)^{(p)}\st{\sim}{\lra} I_{-}(\u_-^{(p)}).\]
The condition $\varphi: Fr_X^\ast\theta_\nabla\simeq \psi_\nabla$ means that $\theta_\nabla$ is mapped to $\psi_\nabla$ under the induced morphism
\[H^0(X, I_{+}(\u_-)\otimes\Omega_X^1)\lra H^0(X, Fr_X^\ast I_{+}(\u_-)\otimes\Omega_X^1)\st{\sim}{\lra}H^0(X, I_{-}(\u_-^{(p)})\otimes\Omega_X^1). \] One checks easily that this agrees with the definition in subsection \ref{subsection explicit} when $G=\GL_n$.

Equivalently, we can give a Tannakian definition: a de Rham $G$-gauge of type $\mu$ over $X$ is an exact $\F_p$-linear tensor functor 
\[\mathfrak{I}: \Rep\,G\ra F\tr{-Gauge}_{dR}(X)\] of type $\mu$. This means that the induced functor $\mathfrak{z}: \Rep\,G\ra F\tr{-Zip}(X)$ (by composing $\mathfrak{I}$ with the forgetful functor $F\tr{-Gauge}_{dR}(X)\ra F\tr{-Zip}(X)$) has type $\mu$, cf. subsection \ref{subsection F-zip}. 

By construction,
we have the natural forgetful functor \[F\tr{-Gauge}_{dR}^{G,\mu}(X)\lra G\tr{-Zip}^\mu(X).\]
In the Tannakian formulation, we can define a larger category $F\tr{-Gauge}_{dR}^{G}(X)$ by allowing $\mu$ varies. Then one has a decomposition of $F\tr{-Gauge}_{dR}^{G}(X)$ according to different types, similar to the case of $G$-zips as in \cite{PWZ15} subsection 7.1. The forgetful functor extends to $F\tr{-Gauge}_{dR}^{G}(X)\lra G\tr{-Zip}(X)$.

\begin{remark}
		Assume that $\mu$ is minuscule. In \cite{Dri23} Drinfeld studied moduli stacks of similar objects as Definition \ref{def G-gauge}, see Theorem \ref{thm Drinfeld}. Compared with here, except the above difference on sign,  there are 2 subtleties to mention: first, in \cite{Dri23} the cocharacter $\mu$ is also defined over $\F_p$, so that there is no Frobenius twist on $P_-$; next, instead of an isomorphism $\varphi$ in the last term, Drinfeld requires an equality (called the Katz condition) between $\theta_\nabla$ and $\psi_\nabla$ by a careful choice of sign\footnote{Such a choice of sign is also included in the definition of the Cartier and inverse Cartier transforms $C_{\X_2/W_2}$ and  $C_{\X_2/W_2}^{-1}$  in \cite{OV}.}: $\psi_\nabla=-\theta_\nabla$.
\end{remark}

\subsection{Moduli stacks of de Rham $F$-gauges with $G$-structure}\label{subsection moduli dR gauges}
In this subsection, assume that $\mu$ is \emph{minuscule}. 
We consider the moduli of de Rham $F$-gauges with $(G,\mu)$-structure.  This has already been studied by Drinfeld (see also \cite{GM}).
	\begin{theorem}[Drinfeld \cite{Dri23}]\label{thm Drinfeld}
		The functor $X\mapsto F\tr{-Gauge}_{dR}^{G,\mu}(X)$ is represented by a smooth algebraic stack $BT^{G,\mu}_1$ over $\kappa$. The natural morphism \[BT^{G,\mu}_1\ra G\tr{-Zip}^\mu\] is an fppf gerbe banded by a finite flat commutative group scheme $Lau_1^G$.
	\end{theorem}
	More precisely, we only defined $BT^{G,\mu}_1(X)=F\tr{-Gauge}_{dR}^{G,\mu}(X)$ for smooth schemes $X/\kappa$. By \cite{Dri23} Lemma 4.1.2, Remark 4.1.3 and Appendix A, $BT^{G,\mu}_1$ is uniquely determined by its restriction to the category of smooth schemes over $\kappa$, in the sense it is the sheafified left Kan extension of its restriction to the category of smooth schemes.
	Over a scheme $X\ra G\tr{-Zip}^\mu$ over $G\tr{-Zip}^\mu$, the finite flat group scheme $Lau_1^G$ is the finite group scheme of height 1 corresponding to the commutative restricted Lie $\Ol_X$-algebra $I_+(\u_-)$ with $p$-operation coming from the composition
	\[Fr_X^\ast I_{+}(\u_-)\st{\sim}{\lra} I_{-}(\u_-^{(p)})\hookrightarrow I_-(\g^{(p)})\simeq I_-(\g)\simeq I_+(\g)\twoheadrightarrow I_+(\u_-). \]
	In particular,
	over $k=\ov{\F}_p$, the finite group scheme $Lau_1^G$ has Lie algebra $\Lie\,Lau_1^G=\u_-$.
	As $X$ varies, these group schemes are compatible under base change, which define a group scheme over $G\tr{-Zip}^\mu$, still denoted by $Lau_1^G$. For each scheme $X/\kappa$ and a morphism $f: X\ra G\tr{-Zip}^\mu$, the fiber of $BT^{G,\mu}_1(X)\ra G\tr{-Zip}^\mu(X)$ at $f\in G\tr{-Zip}^\mu(X)$ is the groupoid of $Lau_1^G$-torsors over $X$.
	
	In particular, the underlying topological spaces of the two moduli stacks are homeomorphic: \[| BT^{G,\mu}_1|\simeq |G\tr{-Zip}^\mu|.\] As in subsection \ref{subsection F-zip}, the latter was known by the works of Pink-Wedhorn-Ziegler ( \cite{PWZ11} and \cite{PWZ15}): \[|G\tr{-Zip}^\mu|\simeq ({}^{J}W, \preceq),\] where $J\subset W$ is the set of simple generators defined by $\mu$, ${}^{J}W$ is the set of minimal length representatives for $W_J\backslash W$.
		
		The stack $BT^{G,\mu}_1$ should be the special fiber of a smooth algebraic stack living over $\Ol_E$, see \cite{Dri23} Conjecture C.3.1 (which has been proved by Gardner-Madapusi, cf. \cite{GM} Theorem A).
		If $G=\GL_h$ and $\{\mu\}$ is given by $(1^{d}, 0^{h-d})$ for some integer $0\leq d\leq h$, Drinfeld also conjectures that the stack $BT^{G,\mu}_1$ is isomorphic to the stack of 1-truncated $p$-divisible groups of height $h$ and dimension $d$, cf. \cite{Dri23} Conjectures 4.5.2 and 4.5.3. Recently, Gardner and Madapusi  announced a theorem which says that there is an equivalence between the category of  truncated $p$-divisible groups of height $h$ and dimension $d$ and the category $F\tr{-Gauge}_{dR}^{G,\mu}(X)$ (see \cite{GM} Conjecture 1 and the paragraph below it).  In particular, this equivalence implies Drinfeld's conjecture.

\section{Applications to  Shimura varieties}
We apply the theory of de Rham $F$-gauges to the study of good reductions of Shimura varieties. On the one hand, we construct universal de Rham $F$-gauges on the reduction mod $p$ of smooth integral canonical models for Shimura varieties of Hodge type. On the other hand, we study the cohomology of de Rham $F$-gauges on these varieties by dual BGG complexes.

\subsection{Automorphic vector bundles in characteristic $p$}\label{subsection auto vb}
Let $(\mathbf{G, \mathfrak{X}})$ be a Shimura datum of Hodge type and $p$ a prime such that $\mathbf{G}$ is unramified at $p$.
Then we have a reductive integral model $\mathbf{G}_{\Z_p}$ of $\mathbf{G}_{\Q_p}$.
 Fix a sufficiently small open compact subgroup $K^p\subset \mathbf{G}(\A_f^p)$ and set $K=K^p\mathbf{G}_{\Z_p}(\Z_p)$. By \cite{Kis} and \cite{KM} we have the smooth integral canonical model $\mathscr{S}_K$ over $\Ol_E$ of the Shimura variety $\Sh_K$ over the local reflex field $E$. By \cite{MP} and \cite{Lan13}, we have smooth toroidal compactifications $\mathscr{S}_{K,\Sigma}^{\tr{tor}}$.
Let $X=\mathscr{S}_K\otimes_{\Ol_E}\ov{\F}_p$ (resp. $X^{\tr{tor}}=\mathscr{S}_{K,\Sigma}^{\tr{tor}}\otimes_{\Ol_E}\ov{\F}_p$) be geometric special fiber of $\mathscr{S}_K$ (resp. $\mathscr{S}_{K,\Sigma}^{\tr{tor}}$). Denote $G=\mathbf{G}_{\Z_p}\otimes\F_p$. As in \cite{SZ} 2.2.1, the conjugacy class of the attached Hodge cocharacter has a  representative defined over $\Ol_E$. Let $\kappa$ be the residue field of $\Ol_E$, $\mu$ the induced cocharacter of $G$ over $\kappa$, and $\Fl_{G,\mu}$ the associated flag variety over $k:=\ov{\kappa}=\ov{\F}_p$.


Recall that we have the following diagram of schemes over $k$
\[ \xymatrix{ & \wt{X}\ar[ld]_\tau \ar[rd]^q&\\
	X & & \Fl_{G,\mu},
	}\]
	where $\tau$ is a $G$-torsor and $q$ is $G$-equivariant. This is the geometric special fiber of  the diagram (in the good reduction case) in \cite{KP} Theorem 4.2.7. In terms of the following subsection, this diagram corresponds to the pair $(I, I_+)$ with $\wt{X}=I$ and $q$ induced by the $P_\mu$-torsor $I_+$.
Let $P=P_\mu=P_+$ with associated Levi $L$. Since $\Fl_{G,\mu}\cong G/P$ as schemes over $k$, it is well known that there is an equivalence of categories \[\Rep\,P\simeq \Vect^G(\Fl_{G,\mu}),\] where the later is the category of $G$-equivariant vector bundles on  $\Fl_{G,\mu}$. Then
from the above diagram, we get the functor
\[\E(\cdot): \Rep\,P\lra \Vect(X).\]
The objects in the essential image are called autormorphic vector bundles on $X$. For any representation $V$ of $G$, viewed as a representation of $P$ by the inclusion $P\subset G$, the associated vector bundle $\V=\E(V)$ admits a natural flat connection $\nabla$. On the other hand, for any representation $W$ of $L$, viewed as a representation of $P$ via the projection $P\ra L$,  we have also the associated (semisimple) vector bundle $\mathcal{W}$ over $X$.

Considering the pullback of the above diagram under the Frobenius $Fr_k: \Spec\,k\ra \Spec\,k$, we get a commutative diagram
\[\xymatrix{ \Rep\,P\ar[r]^{\E(\cdot)}\ar[d]_{\pi^\ast_P}& \Vect(X)\ar[d]^{\pi^\ast_X}\ar[r]^{Fr_X^\ast} & \Vect(X).\\
	\Rep\,P^{(p)}\ar[r]^{\E^{(p)}(\cdot)}& \Vect(X')\ar[ru]_{F_{X/k}^\ast} &
	} \]
In other words, we have a natural isomorphism of functors
\[ Fr_X^\ast\circ\E(\cdot)\simeq F_{X/k}^\ast\circ \E^{(p)}(\cdot)\circ \pi^\ast_P.\]
On the other hand, we have also an isomorphism of functors
$\E(\cdot)\circ Fr_P^\ast\simeq Fr_X^\ast\circ\E(\cdot)$. Let $Fr_{P/k}: P\ra P^{(p)}$ be the relative Frobenius morphism for $P$. Then $Fr_P=\pi_P\circ Fr_{P/k}$. For $V\in \Rep\,P$, we get $Fr_P^\ast V=Fr_{P/k}^\ast\circ \pi_P^\ast V= Fr_{P/k}^\ast(V^{(p)})\in \Rep\,P$, the Frobenius twisted representation.

Fix $T\subset B\subset P$ a maximal torus and Borel subgroup inside $P$. Let $\lambda\in X^\ast(T)_+$ (resp. $X^\ast(T)_{L,+}$) be the associated set of dominant characters ($L$-dominant characters). For any $\lambda\in X^\ast(T)_+$, we get the representation $V_\lambda$ of $G$ with highest weight $\lambda$, and the associated flat vector bundle $(\mathcal{V}_\lambda, \nabla)$ over $X$.  On the other hand,
 for $\eta\in X^\ast(T)_{L,+}$, we have also the associated (semisimple) vector bundle $\mathcal{W}_\eta$ over $X$.
 
 \subsection{De Rham $F$-gauges on good reductions of Shimura varieties}
 
 By works of Moonen-Wedhorn \cite{MW}, C. Zhang \cite{Zhang2018EO}, Goldring-Koskivirta \cite{GoldringKoskivirta2019}, Lan-Stroh \cite{LS18}, and Andreatta \cite{And22}, there exits a smooth morphism of algebraic stacks
 \[\zeta: X\lra G\tr{-Zip}^\mu.\]
 We get the induced Ekedahl-Oort stratification  \[X=\coprod_{w\in {}^{J}W}X^w.\] 
 Via the theory of generalized Hasse invariants, this is of key importance to understand the coherent cohomology of $X$ and construction of automorphic Galois representations in irregular and torsion case, cf. \cite{GoldringKoskivirta2019}. 
 
 The morphism $\zeta$ corresponds to a universal $G$-zip $(I, I_+, I_-, \varphi)$ of type $\mu$ over $X$, which we briefly review. Fix a symplectic embedding $i: (\mathbf{G, \mathfrak{X}})\hookrightarrow (\mathrm{GSp}(V,\psi), S^\pm)$ and a lattice $V_{\Z_p}\subset V_{\Q_p}$ such that we have a closed embedding $\mathbf{G}_{\Z_p}\subset \GL(V_{\Z_p})$. Fix
  a finite collection of tensors $s_\alpha\in V^\otimes_{\Z_p}$ defining $\mathbf{G}_{\Z_p}\subset \GL(V_{\Z_p})$. Let \[\mathcal{A}\lra X\] be the pull back of the universal abelian scheme over the associated Siegel modular variety. Consider the vector bundle \[\mathcal{V}=H^1_{\dR}(\mathcal{A}/X)\] together with the Gauss-Manin connection $\nabla$ over $X$. We have also the Frobenius and Verschiebung maps $F: \V\ra \V$ and $V: \V\ra \V$.
  The Hodge filtration $C_\bullet$ on $\V$ is given by $\V\supset \tr{Ker}\, F\supset 0$ and the conjugate filtration $D_\bullet$ on $\V$ is given by $0\subset \tr{Im}\, V \subset \V$. We get a de Rham $F$-gauge \[(\V, \nabla, C^\bullet, D_\bullet, \varphi)\] by Example \ref{example F-gauge}.
   By \cite{Kis} Corollary 2.3.9 and \cite{KM} Proposition 4.8, we have de Rham tensors $s_{\alpha,\dR}\in \V^\otimes$. Set $M=(V_{\Z_p}\otimes\F_p)^\vee$. Then the cocharacter $\mu$ induces a decreasing filtration $M\supset M^1\supset 0$ and the cocharacter $\mu^{(p)}$ induces an increasing filtration $0\subset M_0\subset M$.
  The $G$-zip $(I, I_+, I_-, \varphi)$  of type $\mu$ over $X$ is defined by \begin{itemize}
  \item $I=\Isom_X\Big((M, s_\alpha)\otimes\Ol_X, (\V, s_{\alpha,\dR})\Big)$,
  \item $I_+=\Isom_X\Big((M\supset M^1, s_{\alpha})\otimes\Ol_X, (\V \supset \tr{Ker}\, F, s_{\alpha,\dR})\Big)$,
  \item $I_-=\Isom_X\Big((M_0\subset M, s_{\alpha})\otimes\Ol_X, (\tr{Im}\, V\subset\V, s_{\alpha,\dR})\Big)$,
  \item $\varphi: I_+^{(p)}/U_+^{(p)}\ra I_-/U_-^{(p)}$ is the isomorphism with precise formula given in \cite{Zhang2018EO} Theorem 2.4.1.
\end{itemize} 
 Here is the first main result of this section, which is motivated by the works of Drinfeld \cite{Dri23, Drin} (in particular, see \cite{Drin} subsection 4.3).
 \begin{theorem}\label{thm univ F-gauge}
 	Let $X=\mathscr{S}_K\otimes_{\Ol_E}\kappa$ be the reduction modulo $p$ of the integral canonical model of a Hodge type Shimura variety with hyperspecial level at $p$.
 	There is a natural de Rham $F$-gauge with $(G, \mu)$-structure over $X$, which induces the universal $G$-zip of type $\mu$ given by $\zeta$. The induced morphism $\xi: X\ra BT^{G,\mu}_1$ is smooth.
 	In other words, we have a commutative diagram of smooth morphisms of algebraic stacks over $\kappa$:
 	\[ \xymatrix{ X\ar[r]^-\xi\ar[rd]_\zeta & BT^{G,\mu}_1\ar[d]\\
 		& G\tr{-Zip}^\mu.
 	}\]
 \end{theorem}
 \begin{proof}
 In fact we have two approaches. First, applying the main result of \cite{Lov} that there exists a universal filtered $F$-crystal on the formal completion of the integral canonical model $\X$ of the involved Shimura variety, we take its reduction modulo $p$ to get (similarly as Proposition \ref{prop MF dR gauge}) the desired exact tensor functor \[\Rep\,G\ra  F\tr{-Gauge}_{dR}(X).\]
 The fact that this construction lifts the universal $G$-zip on $X$ follows from the discussions in \cite{SZ} subsection 5.3. 
 
 Next we explain a direct construction, as an enrichment of the construction of the universal $G$-zip in \cite{Zhang2018EO}. We already have the universal $G$-zip of type $\mu$ given by \[(I, I_+, I_-, \varphi).\] The task is to find a flat connection $\nabla$ on $I$ having all the required properties so that $(I, \nabla, I_+, I_-, \varphi)$ forms a de Rham $F$-gauge of type $\mu$.  By \cite{Kis} and \cite{KM} the de Rham tensors $s_{\alpha,\dR}\in \V^\otimes$ are horizontal with respect to the Gauss-Manin connection $\nabla$ on $\V$. By the definition of $I$ above, we get an induced connection $\nabla$ on $I$. It satisfies the required properties as in Definition \ref{def G-gauge} as the induced flat vector bundle $(I(M), \nabla)=(\V,\nabla)$ together with $I_+(M)$ and $I_-(M)$ forms a de Rham $F$-gauge $(\V,\nabla, C^\bullet,D_\bullet,\varphi)$. Therefore we get a morphism of algebraic stacks
 \[\xi: X\ra BT^{G,\mu}_1,\]which makes
 the following diagram commute
 \[\xymatrix{
 	X\ar[r]^{\xi} \ar[rd]_{\zeta} & BT^{G,\mu}_1\ar[d]^\pi\\
 	& G\tr{-Zip}^\mu.
 } \]
 
 We already know that $\zeta: X\ra G\tr{-Zip}^\mu$ is smooth by \cite{Zhang2018EO} Theorem 4.1.2. 
 The morphism $\pi: BT^{G,\mu}_1\ra G\tr{-Zip}^\mu$ is a gerbe
 banded by $Lau_1^G$, thus it is a smooth and universal homeomorphism (cf. \cite{Sta} Tags 0DN8 and 06R9). 
 The morphism $\xi: X\ra BT^{G,\mu}_1$ is equivalent to a section over $X$ of the morphism $BT^{G,\mu}_1\ra G\tr{-Zip}^\mu$, which can be checked to be smooth. More precisely, let $X^\#\ra X$ be the pullback of the $E_{G,\mu}$-torsor $G\ra G\tr{-Zip}^\mu$, we get a smooth morphism of schemes $\zeta^\#: X^\#\ra G$.  Let $\wt{X^\#}\ra X^\#$ (resp. $\wt{G}\ra G$) be the pullback of the morphism $BT^{G,\mu}_1\ra G\tr{-Zip}^\mu$ under the induced map $X^\#\ra G\tr{-Zip}^\mu$ (resp. $G\ra G\tr{-Zip}^\mu$). In other words, we have the following commutative diagram
 \[\xymatrix{
 	\wt{X^\#}\ar[d]\ar[r]^{\widetilde{\zeta^\#}}&\wt{G}\ar[d]\ar[r]& BT^{G,\mu}_1\ar[d]\\
 	X^\#\ar[d]\ar[r]^{\zeta^\#}& G\ar[r]\ar[d] & G\tr{-Zip}^\mu\\
 	X\ar[r]^\zeta& G\tr{-Zip}^\mu\ar@{=}[ru]& 
 	}\]
 	with all squares cartesian.
 Then the morphism $\xi: X\ra BT^{G,\mu}_1$ is obtained by quotient by $E_{G,\mu}$ of \[\xi^\#: X^\#\ra \wt{G},\] the composition of the induced smooth morphism $\widetilde{\zeta^\#}: \wt{X^\#}\ra \wt{G}$ and a section $s: X^\#\ra \wt{X^\#}$ of the gerbe $\wt{X^\#}\ra X^\#$ banded by $Lau_1^G$. In other words, the gerbe $\wt{X^\#}\ra X^\#$ splits and $\wt{X^\#}\simeq [X^\#/L]$ where $L$ is the Lau group scheme over $X^\#$. The section $s: X^\#\ra \wt{X^\#}$  is thus the canonical $L$-torsor over $\wt{X^\#}$. Although $s$ is only flat, we claim that the composition $\xi^\#=\widetilde{\zeta^\#}\circ s$ is smooth, which is equivalent to that $\xi$ is smooth. Roughly, we will apply the infinitesimal criterion for smoothness.

 To show that $\xi$ is smooth, since both $X$ and $BT^{G,\mu}_1$ are smooth over $\kappa$,
 it suffices to show the following statement: for any algebraically closed field $k|\ov{\F}_p$ and any closed point $x\in X(k)$ with image $y: \Spec\,k\ra BT^{G,\mu}_1$, the induced map between the tangent spaces \[T_x\xi: T_xX\ra T_yBT^{G,\mu}_1\] is surjective. Here $T_yBT^{G,\mu}_1$ is the tangent space of the stack $BT^{G,\mu}_1$ at $y$, which is by definition $\pi_0\tr{Fib}\Big(BT^{G,\mu}_1(k[\varepsilon])\ra BT^{G,\mu}_1(k), y\Big)$, the associated set of isomorphism classes of the groupoid of the fiber of $BT^{G,\mu}_1(k[\varepsilon])\ra BT^{G,\mu}_1(k)$ at $y$, see \cite{Sta} Tags 07WY and 07X1. (As usual, $k[\varepsilon]=k[t]/(t^2)$ is the ring of dual numbers over $k$.)
 This (generalized) criteria for smoothness using tangent spaces seems to be well known (for example see the proof of Theorem 4.5 in \cite{Lau}), which can be deduced easily using the more advanced version of smoothness criteria by cotangent complexes (cf. \cite{Khan} Proposition 8.7.1).
 
 Base changing the above diagram over $k$, we may assume $k=\ov{\F}_p$. Let $R_G=\wh{\Ol}_{X,x}$. Then we have $R_G\simeq \wh{U_-}$, the formal completion of the unipotent group $U_-$ along its identity section. The ring $R_G$ admits a description as a characteristic $p$ deformation space of $p$-divisible groups with crystalline Tate tensors $(s_\alpha)$, cf. \cite{Kis} subsections 1.5 and 2.3. One can interpret $\Spf\,R_G$ as a deformation space (in characteristic $p$) of prismatic $F$-gauges with $(G,\mu)$-structure by the recent work of Ito \cite{Ito}. On the other hand, for each $n\geq 1$, let $BT^{G,\mu}_n$ be the stack over $k$ of defined by Drinfeld in \cite{Dri23} Appendix C, and let $BT^{G,\mu}_\infty=\varprojlim_nBT^{G,\mu}_n$. By \cite{GM} Theorem A, for each $n$, $BT^{G,\mu}_n$ is a smooth Artin stack and the transition morphism $BT^{G,\mu}_{n+1}\ra BT^{G,\mu}_n$ is smooth. The Dieudonn\'e module with its induced crystalline Tate tensors attached to the point $x$ defines a $k$-point $z\in BT^{G,\mu}_\infty(k)$. Consider
 the associated local formal deformation space (over $k$) $\mathrm{Def}_z$ of  $BT^{G,\mu}_\infty$ at $z$. More precisely, for any complete Noetherian $k$-algebra $R$ with maximal ideal $\mathfrak{m}$ such that $k\simeq R/\mathfrak{m}$, $\mathrm{Def}_z(R)=\pi_0\tr{Fib}\Big(BT^{G,\mu}_\infty(R)\ra BT^{G,\mu}_\infty(k),z\Big)$.
 By the work of Ito \cite{Ito}, we have an isomorphism \[\Spf\,R_G\simeq \mathrm{Def}_z.\]  (For more details, see \cite{IKY} Proposition 3.32 and Theorem 3.6.) Hence we get an induced isomorphism between their tangent spaces $T\Spf\,R_G\simeq T\mathrm{Def}_z$ (cf. \cite{Sta} Tag 06IG).
 The map $T_x\xi: T_xX\ra T_yBT^{G,\mu}_1$ is then the composition of \[T_xX=T\Spf\,R_G\simeq T\mathrm{Def}_z\ra T_yBT^{G,\mu}_1,\] which is surjective since the natural map $T\mathrm{Def}_z\ra T_yBT^{G,\mu}_1$ is surjective by the smoothness of $BT^{G,\mu}_{n+1}\ra BT^{G,\mu}_n$, $\forall\,n\geq 1$.
 This finishes the proof.
 \end{proof}
 
 \begin{remark}
 	\begin{enumerate}
 \item The above map between the tangent spaces $T_xX\ra T_yBT^{G,\mu}_1$ is actually a bijection, as both sides are isomorphic to $\Lie\, U_-=\u_-$.
 \item As in Remark \ref{rmk log F-gauge}, over a smooth toroidal compactification $X^{\tr{tor}}$ of $X$, one can in fact show that there is a universal log de Rham $F$-gauge with $G$-structure over $X^{\tr{tor}}$, as a canonical extension of the one over $X$.
\item We leave it to the reader to extend Theorem \ref{thm univ F-gauge} to the abelian type case. For the version without compactification, see \cite{Lov, SZ} for some ideas.
\item See \cite{IKY} Theorem 3.34 for a similar (but more complicated) result in the mixed characteristic setting.
	\end{enumerate}
\end{remark} 
 By Theorem \ref{thm univ F-gauge} and the Tannakian definition of de Rham $F$-gauge with $G$-structure, we have a functor \[\Rep\,G\lra F\tr{-Gauge}_{dR}(X).\] Therefore, for any $\lambda \in X^\ast(T)_+$ such that $V_\lambda$ is defined over $\F_p$, we have an enrichment of the flat vector bundle $(\mathcal{V}_\lambda, \nabla)$  into a de Rham $F$-gauge
 \[(\mathcal{V}_\lambda, \nabla, C^\bullet, D_\bullet, \varphi).\]
 In the following, we will study the $F$-gauge (= $F$-zip) structure on the de Rham cohomology groups $H^i_{\dR}(X/k, (\V_\lambda,\nabla))$.

\subsection{Standard complexes in characteristic $p$}\label{subsection std complexes}
We need some group and representation theoretic preparations in this and the next subsection.
We work over $k=\ov{\F}_p$.
Let $\mathfrak{g}=\Lie\,G, \,\mathfrak{p}=\Lie\,P,\, \mathfrak{u}=\Lie\,U$ and $\mathfrak{u}_-=\Lie\,U_{-}$. For $W\in\Rep\,P$, the associated Verma module is the $U(\g)-P$-module
\[\Verm(W)=U(\g)\otimes_{U(\p)}W, \]
where $U(\g)$ acts canonically on the first component and $P$ acts canonically and diagonally on both components. 
By the Poincar\'e-Birkhoff-Witt theorem, we have a canonical isomorphism
\[ U(\g)\otimes_{U(\p)}W\simeq U(\u_-)\otimes_kW.\]
Since $\u_-$ is abelian, we have a canonical isomorphism
\[ U(\u_-)\simeq \Sym(\u_-),\]
which can be identified with a polynomial algebra over $k$ with variables given by any $k$-basis of $\u_-$. Thus we have \[\Verm(W)\simeq \Sym(\u_-)\otimes_k W.\]
In the following we will omit the subscript $k$ of $\otimes_k$ for tensor products of $k$-vector spaces.
For any integer $m\geq 0$, let $\Sym_{\leq m}(\u_-)$ be the sub $k$-algebra of $\Sym(\u_-)$ of elements of degree $\leq m$, which we identify with a sub $k$-algebra of $U(\u_-)$. Similarly we have $\Sym_{ m}(\u_-)$, the sub $k$-algebra of $\Sym(\u_-)$ of elements of degree $m$.
For $W_1, W_2\in \Rep\,P$, we have a natural isomorphism
\[\Hom_{U(\g)-P}(\Verm(W_1), \Verm(W_2))\simeq \Hom_{P}(W_1,\Verm(W_2)). \]
A morphism $f\in \Hom_{U(\g)-P}(\Verm(W_1), \Verm(W_2))$ is said of degree $m$ if the induced morphism \[f: W_1\ra \Verm(W_2)=\Sym(\u_-)\otimes W_2\] has image in $\Sym_m(\u_-)\otimes W_2$.

Let $f \in \Hom_{U(\g)-P}(\Verm(W_1), \Verm(W_2))$ be a morphism of degree $m$, which we view as a map
\[f: W_1\lra \Sym_{\leq m}(\u_-)\otimes W_2. \]
Passing to dual,  it induces a morphism
\[f^\vee: \Gamma_{\leq m}(\u_-^{\vee})\otimes W_2^\vee\lra W_1^\vee,  \]
where $\Gamma_{\leq m}(\u_-^{\vee})=\Gamma(\u_-^{\vee})/\Gamma_{>m}(\u_-^{\vee})$ is the quotient of the divided power algebra $\Gamma(\u_-^{\vee})$ associated to $\u_-^{\vee}$. By construction, $\Gamma(\u_-^{\vee})$ sits in a canonical perfect pairing
\[\Sym(\u_-)\times \Gamma(\u_-^{\vee})\lra k. \]
There is a degree-preserving canonical morphism
\[\Sym(\u_-^{\vee})\lra \Gamma(\u_-^{\vee}), \]
such that the induced morphism
\[ \Sym_{\leq m}(\u_-^{\vee})\lra  \Gamma_{\leq m}(\u_-^{\vee})\]
is an isomorphism for $m\leq p-1$. Later, we will mainly work with $m=1$ to produce differential operators of degree one, therefore we can use either $\Sym_{\leq 1}(\u_-^{\vee})$ or $\Gamma_{\leq 1}(\u_-^{\vee})$ without changing the construction.

Let $V\in \Rep\,G$ be a representation, which we consider as an object in $\Rep\,P$ by the natural restriction functor $\Rep\,G\ra \Rep\,P$. Let $\rho$ be the half of sum of positive roots.
Set $n=\dim_k\,\u_-$. Consider the standard complex of $U(\g)-P$-modules \[\tr{Std}_\bullet(V)=\Verm(\wedge^\bullet(\u_-)\otimes V)=\Sym(\u_-)\otimes\wedge^\bullet(\u_-)\otimes V,\] more explicitly
\[0\ra \Verm(\wedge^n(\u_-)\otimes V)\ra \Verm(\wedge^{n-1}(\u_-)\otimes V)\ra \cdots \ra \Verm(\u_-\otimes V)\ra \Verm(V)\ra 0, \]
where for $1\leq a\leq n$, the differentials
\[ d_a:  \Verm(\wedge^a(\u_-)\otimes V)\lra \Verm(\wedge^{a-1}(\u_-)\otimes V)\]
are given by
\[\begin{split}
d_a(u\otimes (x_1\wedge x_2\wedge\cdots  \wedge x_a )\otimes v)&=\sum_{i=1}^a(-1)^{i-1}(ux_i)\otimes (x_1\wedge x_2\wedge\cdots \wedge \wh{x_i}\wedge \cdots\wedge x_a)\otimes v\\
&+\sum_{i=1}^a(-1)^{i}u\otimes (x_1\wedge x_2\wedge\cdots \wedge \wh{x_i}\wedge\cdots \wedge x_a)\otimes( x_iv),
\end{split}\]
for all $u\in\Sym(\u_-),\, x_1,\dots, x_a\in \u_-$ and $v\in V$.

Recall that $P=P_\mu$ for a cocharacter $\mu: \G_m\ra G$ defined over $\kappa$, as in the setting of subsection \ref{subsection F-zip}. For $V$ as above, we have a natural decreasing Hodge filtration $C^\bullet$ which is defined as usual by the weights of $\G_m$ via $\mu$: for $a\in\Z$, \[C^a(V)=\bigoplus_{b\geq a} V^b,\] where for each $b\in\Z$, $V^b\subset V$ is the subspace of weight $b$ by the induced action of $\G_m$. It follows that each graded piece $\gr^a_{C}V$ forms a representation of $L$, the Levi subgroup associated to $P$. As any element of $\u_-$ has weight $-1$, we have 
\[d(\Sym(\u_-)\otimes\wedge^j(\u_-)\otimes C^{a-j}(V))\subset \Sym(\u_-)\otimes\wedge^{j-1}(\u_-)\otimes C^{a-j+1}(V).\]
We get a decreasing Hodge filtration $C^\bullet$ on $\tr{Std}_\bullet(V)$ by
\[C^a(\tr{Std}_\bullet(V)):=\Sym(\u_-)\otimes \wedge^{\bullet}(\u_-)\otimes C^{a-\bullet}(V),\quad a\in\Z. \]
The associated graded complex has the form
\[\gr_{C}^a(\tr{Std}_\bullet(V)) = \Sym(\u_-)\otimes \wedge^{\bullet}(\u_-)\otimes \gr_{C}^{a-\bullet}(V),\quad a\in\Z.\]


Now consider the Frobenius twisted parabolic $P^{(p)}\subset G_{\kappa}$ and its opposite parabolic $P^{(p)}_-$. The Lie algebra of the unipotent radical of $P^{(p)}_-$ is $\u_-^{(p)}$. As above, we can consider the category of $U(\g)-P^{(p)}$-modules.
 For any $V\in \Rep\,G$, we repeat the above definition of standard complex $\tr{Std}_\bullet(V)$ to define a complex of $U(\g)-P^{(p)}$-modules.
 \begin{definition}
 	For any $V\in \Rep\,G$, let
   $p\tr{-Std}_\bullet(V)$ be the following complex
\[0\ra \Sym(\u_-^{(p)})\otimes\wedge^n(\u_-^{(p)})\otimes V\ra \cdots \ra \Sym(\u_-^{(p)})\otimes\u_-^{(p)}\otimes V\ra \Sym(\u_-^{(p)})\otimes V\ra 0, \]
where for $1\leq a\leq n$, the differentials
\[ \psi_a:  \Sym(\u_-^{(p)})\otimes\wedge^a(\u_-^{(p)})\otimes V\lra \Sym(\u_-^{(p)})\otimes\wedge^{a-1}(\u_-^{(p)})\otimes V\]
are given by \[\begin{split}\psi_a(u\otimes (x_1\wedge x_2\wedge\cdots   \wedge x_a)\otimes v)&=\sum_{i=1}^a(-1)^{i-1}(ux_i)\otimes (x_1\wedge x_2\wedge\cdots \wedge \wh{x_i}\wedge \cdots\wedge x_a)\otimes v\\
&+\sum_{i=1}^a(-1)^{i}u\otimes (x_1\wedge x_2\wedge\cdots \wedge \wh{x_i}\wedge\cdots \wedge x_a)\otimes( x_iv),
\end{split}\] 
for all $u\in \Sym(\u_-^{(p)}),\, x_1,\dots, x_a\in \u_-^{(p)}$ and $v\in V$.
\end{definition}

Write $P_-^{(p)}=P_\nu$. We can define similarly an increasing conjugate filtration $D_\bullet$ on $V$ by weights of $\G_m$-action via $\nu: \G_m\ra G_\kappa\simeq G_\kappa^{(p)}$: \[D_a(V)=\bigoplus_{b\leq a}V^b.\] Note that $\nu=\mu^{\sigma,-1}$ with $\sigma=Id_G\times Fr_k: G_k\ra G_k$ the Frobenius of $G$ over $k$. In particular, if $V$ is defined over $\F_p$,  we get an $F$-zip  \[(V, C^\bullet, D_\bullet, \varphi)\] using the isomorphisms $V^{(p)}\simeq V$ and $G^{(p)}\simeq G$. 
Recall here $\varphi$ is given by isomorphisms
\[\varphi_a: Fr_k^\ast\gr^a_{C}V=(\gr^a_{C}V)^{(p)} \st{\sim}{\lra}\gr^a_{D}V, \quad a\in\Z.\]
Now, as any element of $\u_-^{(p)}$ has weight 1, we have
\[\psi_a(\Sym(\u_-^{(p)})\otimes\wedge^a(\u_-^{(p)})\otimes D_i(V) )\subset \Sym(\u_-^{(p)})\otimes\wedge^{a-1}(\u_-^{(p)})\otimes D_{i+1}(V). \]
Thus we can define an increasing conjugate filtration $D_\bullet$ on $p\tr{-Std}_\bullet(V)$ by
\[D_a(p\tr{-Std}_\bullet(V))=\Sym(\u_-^{(p)})\otimes \wedge^\bullet(\u_-^{(p)})\otimes D_{a-\bullet}(V), \quad a\in\Z. \]
The associated graded complex has the form
\[ \gr_{D}^a(p\tr{-Std}_\bullet(V)) = \Sym(\u_-^{(p)})\otimes \wedge^{\bullet}(\u_-^{(p)})\otimes \gr_{D}^{a-\bullet}(V),\quad a\in\Z.\] 
\begin{lemma}\label{lemma grad}
Let $V$ be a representation of $G$ defined over $\F_p$. 
The isomorphisms $\varphi_a: (\gr^a_{C}V)^{(p)} \st{\sim}{\lra}\gr^a_{D}V$ extend to isomorphisms
\[\varphi_a: (\gr_{C}^a(\tr{Std}_\bullet(V)))^{(p)}\simeq  \gr_{D}^a (p\tr{-Std}_\bullet(V)). \]
\end{lemma}
\begin{proof}
By the above description, the corresponding terms of the two graded complexes on both sides are isomorphic. One can check that the differentials on both sides are also isomorphic under $\varphi_a$.
\end{proof}

\subsection{BGG complexes in characteristic $p$}\label{subsection BGG}
The construction of BGG complexes in mixed characteristic is due to Polo-Tiluilne \cite{PT} and Lan-Polo \cite{LP}. We review the construction in the Hodge type setting, with emphasis on characteristic $p$. We remark that although the main text of \cite{LP} is restricted to the PEL setting, the construction of BGG complexes works more generally, which can be seen from the proof of Theorem 5.2 in \cite{LP} (improving the construction of \cite{PT} in the setting of split reductive groups with simply connected derived group). 

Keep the notations as above.
Let $\Phi$ be the set of roots of $G$ and $\Phi^+$ the set of positive roots corresponding to the choice of the Borel subgroup $B$. We recall a crucial notion for modular representations.
A weight $\lambda\in X^\ast(T)_+$ is called $p$-small, if \[\lan \lambda+\rho, \alpha^\vee\ran \leq p, \quad \forall\,\alpha\in \Phi,\] where as before $\rho=\frac{1}{2}\sum_{\alpha\in \Phi^+}\alpha$.
Equivalently, $\lambda\in X^\ast(T)_+$ is $p$-small if \[\lambda\in X^\ast(T)_+ \cap \ov{C}_p,\] where $\ov{C}_p$ is the closure of the fundamental $p$-alcove
\[\ov{C}_p=\{\lambda\in X^\ast(T)_\R\,|\,0\leq \lan \lambda+\rho, \alpha^\vee\ran \leq p,\,\forall\,\alpha\in \Phi^+\} \]
for the dot action of the affine Weyl group $\mathbf{W}_{\tr{aff}}\simeq p\Z\Phi\rtimes W$ on $X^\ast(T)_\R$ (see \cite{PT} subsection 1.9 and the references therein).
Here the dot action is given by $w\cdot\lambda=w(\lambda+\rho)-\rho$. 
Consider the Coxeter number of $G$
\[h:=1+\tr{Max}\{\lan \rho, \alpha^\vee\ran, \alpha\in \Phi^+\}.\]  Then one checks easily that \[X^\ast(T)_+ \cap \ov{C}_p\neq\emptyset \quad \Longleftrightarrow \quad p\geq h-1.\]

Recall the subset ${}^JW$ of the Weyl group $W$ associated to $P$ introduced in subsection \ref{subsection moduli dR gauges}. The length function induces a surjection
\[\ell: {}^JW\lra \{0,1,\dots, n\},\quad w\mapsto \ell(w).\]
For each $0\leq a\leq n$, let ${}^JW(a)\subset {}^JW$ be the subset of elements of length $a$. If $\lambda\in X^\ast(T)_+$ and $w\in {}^JW$, then $w\cdot \lambda \in X^\ast(T)_{L,+}$.
\begin{theorem}\label{thm BGG}
	Assume that $\lambda$ is $p$-small. There exists a $C^\bullet$-filtered complex $\tr{BGG}_\bullet(V_\lambda)$ of $U(\g)-P$-modules 
		\[0\ra \tr{BGG}_n(V_\lambda)\ra \cdots \ra \tr{BGG}_1(V_\lambda)\ra \tr{BGG}_0(V_\lambda)\ra 0, \] where 
		\[\tr{BGG}_a(V_\lambda)=\bigoplus_{w\in {}^JW(a)}\Verm(W_{w\cdot\lambda}), \quad 0\leq a\leq n. \]
		Moreover, there is a $C^\bullet$-filtered quasi-isomorphic embedding
		\[ \tr{BGG}_\bullet(V_\lambda) \hookrightarrow \tr{Std}_\bullet(V_\lambda),\] such that the induced map of graded complexes \[\gr_{C}  \tr{BGG}_\bullet(V_\lambda) \hookrightarrow  \gr_{C} \tr{Std}_\bullet(V_\lambda)\] is a quasi-isomorphic direct summand of degree zero\footnote{This means that both the embedding and the splitting morphisms (defining the summand) are defined by direct sums of morphisms of $U(\g)-P$-modules of degree zero, see \cite{LP} Definition 5.1} with trivial differentials.
\end{theorem}
\begin{proof}
This is \cite{LP} Theorem 5.2, which generalizes the Theorem D of \cite{PT}. See also the proof of Proposition \ref{prop p-BGG} below.
\end{proof}

\begin{proposition}\label{prop p-BGG}
	Assume that $\lambda$ is $p$-small and $V_\lambda$ is defined over $\F_p$. There exists a $D_\bullet$-filtered complex $p\tr{-BGG}_\bullet(V_\lambda)$ of $U(\g)-P^{(p)}$-modules 
	\[0\ra p\tr{-BGG}_n(V_\lambda)\ra \cdots \ra p\tr{-BGG}_1(V_\lambda)\ra p\tr{-BGG}_0(V_\lambda)\ra 0, \] where 
	\[p\tr{-BGG}_a(V_\lambda)=\bigoplus_{w\in {}^JW(a)}\Sym(\u_-^{(p)})\otimes W_{w\cdot\lambda}^{(p)}, \quad 0\leq a\leq n. \]
Moreover,	there exists  a
		$D_\bullet$-filtered 
		quasi-isomorphic embedding
		\[ p\tr{-BGG}_\bullet(V_\lambda) \hookrightarrow p\tr{-Std}_\bullet(V_\lambda),\] such that the induced map of graded complexes \[\gr_{D}(p\tr{-BGG}_\bullet(V_\lambda)) \hookrightarrow \gr_{D}(p\tr{-Std}_\bullet(V_\lambda))\] is a quasi-isomorphic direct summand of degree zero with trivial differentials.
\end{proposition}
\begin{proof}
The proof is similar to that of \cite{LP} Theorem 5.2. For the reader's convenience, we review the proof here. Firstly, by assumption $X^\ast(T)_+\cap \ov{C}_p\neq\emptyset$, $0\in X^\ast(T)_+$ is $p$-small, arguing as \cite{PT} 4.4 we have 
\[\wedge^a(\u_-)\simeq \bigoplus_{w\in {}^JW(a)} W_{w\cdot 0}\] as $P$-modules, and hence \[\wedge^a(\u_-^{(p)})\simeq \bigoplus_{w\in {}^JW(a)} W_{w\cdot 0}^{(p)}\] as $P^{(p)}$-modules.
Since $V_0=k$ is defined over $\F_p$, we have natural isomorphism $V_0\simeq V_0^{(p)}$ and $p\tr{-Std}_\bullet(V_0)\simeq (\tr{Std}_\bullet(V_0))^{(p)}$ as complexes of $U(\g)-P^{(p)}$-modules. More explicitly, for any $a\in\Z$,
\[p\tr{-Std}_a(V_0)\simeq \bigoplus_{w\in {}^JW(a)} \Sym(\u_-^{(p)})\otimes W_{w\cdot 0}^{(p)}. \]

Consider a general $\lambda\in X^\ast(T)_+\cap \ov{C}_p\neq\emptyset$ as in the statement of the proposition. Since $V_\lambda$ is defined over $\F_p$, we have $V_\lambda^{(p)}\simeq V_\lambda$. From the isomorphism (cf. \cite{PT} 4.5) $\tr{Std}_a(V_\lambda)\simeq \tr{Std}_a(V_0)\otimes V_\lambda$, we get \[ p\tr{-Std}_a(V_\lambda)\simeq \bigoplus_{w\in {}^JW(a)}\Sym(\u_-^{(p)})\otimes W_{w\cdot 0}^{(p)}\otimes V_\lambda. \]
Denote $p\tr{-Std}_w(V_\lambda)=\Sym(\u_-^{(p)})\otimes W_{w\cdot 0}^{(p)}\otimes V_\lambda$, so that $p\tr{-Std}_\bullet(V_\lambda)\simeq \bigoplus_{w\in {}^JW}p\tr{-Std}_w(V_\lambda)$.
Arguing as \cite{PT} 2.7,
we have a decomposition of subcomplexes of $U(\g)-P^{(p)}$-modules
\[p\tr{-Std}_\bullet(V_\lambda)\simeq\bigoplus_{j\in J}p\tr{-Std}_\bullet(V_\lambda)_{\chi_j}, \]
where $J$ is some finite set, such that the Harish-Chandra part $U(\g)^G\subset Z(\g)$ of the center\footnote{In characteristic $p$, there is also a Frobenius part of the center $Z(\g)$, cf. \cite{Jant} C.4 or \cite{BMR} 3.1.6 and the references therein.} $Z(\g)\subset U(\g)$ acts on $p\tr{-Std}_\bullet(V_\lambda)_{\chi_j}$ by a distinct character $\chi_j^{(p)}$. There is a distinguished character $\chi_\lambda$, which is the unique  character of $U(\g)^G$ which acts non-trivially on $V_\lambda$. Now define 
\[p\tr{-BGG}_\bullet(V_\lambda)= p\tr{-Std}_\bullet(V_\lambda)_{\chi_\lambda}=\bigoplus_{w\in {}^JW}p\tr{-Std}_w(V_\lambda)_{\chi_\lambda},\]
where $p\tr{-Std}_w(V_\lambda)_{\chi_\lambda}=(\Sym(\u_-^{(p)})\otimes W_{w\cdot 0}^{(p)}\otimes V_\lambda)_{\chi_\lambda}$. Then we need to compute each term $p\tr{-Std}_w(V_\lambda)_{\chi_\lambda}$. 

As $\lambda$ is $p$-small, by \cite{PT} Lemma 2.3 all weights of \[\wedge^\bullet(\u_-^{(p)})\otimes V_\lambda=\bigoplus_{w\in {}^JW} (W_{w\cdot 0}^{(p)}\otimes V_\lambda)\] are $p$-small.  For each $w$,
by \cite{PT} Lemma 2.3 and Lemma 1.11,
 there exists a finite filtration on the $P^{(p)}$-module $W_{w\cdot 0}^{(p)}\otimes V_\lambda$ such that the graded pieces are of the form $W_\eta$ for some $\eta\in X^\ast(T)_{L^{(p)},+}\cap \ov{C}_p$.  By the same argument applied to the $P$-module $W_{w\cdot 0}\otimes V_\lambda$, we get $W_\nu$ for some $\nu \in X^\ast(T)_{L,+}\cap \ov{C}_p$. Since \[(W_{w\cdot 0}\otimes V_\lambda)^{(p)}\simeq W_{w\cdot 0}^{(p)}\otimes V_\lambda^{(p)}\simeq W_{w\cdot 0}^{(p)}\otimes V_\lambda,\] by the uniqueness in the statement of \cite{PT} Lemma 1.11, we have \[\eta\simeq \nu^{(p)} \quad \tr{and}\quad W_\eta\simeq W_\nu^{(p)}.\] Then we get also a finite filtration on $p\tr{-Std}_w(V_\lambda)$ by $U(\g)-P^{(p)}$-modules, with graded pieces $\Sym(\u_-^{(p)})\otimes W_\nu^{(p)}$ for those $\nu$ as above. For the direct summand $p\tr{-Std}_w(V_\lambda)_{\chi_\lambda}$, we have a similar finite filtration. For each $\nu$ as above, by \cite{PT} 2.7 and 2.8, $\Sym(\u_-^{(p)})\otimes W_\nu^{(p)}$ appears as a graded piece of the filtration on $p\tr{-Std}_w(V_\lambda)_{\chi_\lambda}=\Sym(\u_-^{(p)})\otimes W_{w\cdot 0}^{(p)}\otimes V_\lambda$ if and only if $W_\nu$ appears as a graded piece of the filtration on $W_{w\cdot 0}\otimes V_\lambda$ and $\nu=w'\cdot \lambda$ for some $w'\in {}^JW$. If these condition hold, by \cite{PT} 4.5 we have exactly $w'=w$ with multiplicity one. Therefore, we have \[p\tr{-Std}_w(V_\lambda)_{\chi_\lambda}=\Sym(\u_-^{(p)})\otimes W_{w\cdot\lambda}^{(p)},\]  and
\[p\tr{-BGG}_a(V_\lambda)=\bigoplus_{w\in {}^JW(a)}\Sym(\u_-^{(p)})\otimes W_{w\cdot\lambda}^{(p)}, \quad 0\leq a\leq n. \]

Let $H\in X_\ast(T)$ be the element defined by the conjugacy class of $\mu$. We define a $D_\bullet$-filtration on $p\tr{-BGG}_\bullet(V_\lambda)$ as 
\[D_a(p\tr{-BGG}_\bullet(V_\lambda))=\bigoplus_{w\in {}^JW\atop w\cdot\lambda(H)\geq -a}\Sym(\u_-^{(p)})\otimes W_{w\cdot\lambda}^{(p)}.\]
Then the
quasi-isomorphic embedding
\[ p\tr{-BGG}_\bullet(V_\lambda) \hookrightarrow p\tr{-Std}_\bullet(V_\lambda)\]  is compatible with the $D_\bullet$-filtrations on both sides, and
the induced map of graded complexes \[\gr_{D}(p\tr{-BGG}_\bullet(V_\lambda)) \hookrightarrow \gr_{D}(p\tr{-Std}_\bullet(V_\lambda))\] is a quasi-isomorphic direct summand of degree zero with trivial differentials, by the same arguments as the last paragraph of the proof of Theorem 5.2 in \cite{LP}. Alternatively, this follows from Lemma \ref{lemma grad} and similar statements in Theorem \ref{thm BGG}.
\end{proof}

\subsection{De Rham complexes and $p$-curvature complexes}
We transfer the results of subsection \ref{subsection std complexes} to geometric setting. In the rest of this section we assume that $X$ is of PEL type. For simplicity, assume moreover that the reductive group defined by the PEL datum is connected (i.e. we exclude the case with local factors of groups of type D) and $X=X^{\tr{tor}}$ is proper\footnote{All the results hold in general PEL type case, by more carefully working with smooth toroidal compactifications, canonical or subcanonical extensions of automorphic vector bundles, connections with log poles, and certain Galois orbit $[\lambda]$ in case with local factors of type D, cf. \cite{LP}}.

For any $\lambda\in X^\ast(T)_+$, we get the representation $V_\lambda^\vee$ of $G$ with highest weight $-w_0(\lambda)$, and the associated flat vector bundle $(\mathcal{V}_\lambda^\vee, \nabla)$ over $X$. Here $w_0$ is the element of maximal length in the Weyl group.
We will assume that $V_\lambda$ is defined over $\F_p$.
Recall the de Rham complex $DR(\V_\lambda^\vee,\nabla)$ and the associated $p$-curvature complex $K(\V_\lambda^\vee,\psi)$. Consider the dual complex $\tr{Std}^\bullet(V_\lambda^\vee)$
\[ 0\ra \Verm(V_\lambda)^\vee\ra \Verm(\u_-\otimes V_\lambda)^\vee\ra \cdots\ra \Verm(\wedge^n(\u_-)\otimes V_\lambda)^\vee\ra 0 \]
and the dual complex $p\tr{-Std}^\bullet(V_\lambda^\vee)$
\[\begin{split}
0\ra \Sym(\u_-^{(p)})^\vee\otimes V_\lambda^\vee\ra \Sym(\u_-^{(p)})^\vee\otimes &(\u_-^{(p)})^\vee\otimes V_\lambda^\vee\ra \cdots\\ &\ra \Sym(\u_-^{(p)})^\vee\otimes(\wedge^n\u_-^{(p)})^\vee\otimes V_\lambda^\vee\ra  0
\end{split}\]
of $\tr{Std}_\bullet(V_\lambda)$ and $p\tr{-Std}_\bullet(V_\lambda)$ respectively. The filtrations $C^\bullet$ and $D_\bullet$ on $V_\lambda^\vee$ naturally induce corresponding filtrations on $\tr{Std}^\bullet(V_\lambda^\vee)$ and $p\tr{-Std}^\bullet(V_\lambda^\vee)$.

Let us make the differentials in $\tr{Std}^\bullet(V_\lambda^\vee)$ more explicit. Consider a map \[f\in \Hom_{U(\g)-P}(\Verm(W_1), \Verm(W_2))\simeq \Hom_{P}(W_1,\Verm(W_2))\] of degree one, i.e. it is given by $f: W_1\ra \Sym_{\leq 1}(\u_-)\otimes W_2$. Its dual is then $f^\vee: \Sym_{\leq 1}(\u_-^{\vee})\otimes W_2^\vee\ra W_1^\vee$. Applying to the differentials in $\tr{Std}_\bullet(V_\lambda)$ (which are all of degree one), in particular \[d_1: \Sym(\u_-)\otimes \u_-\otimes V_\lambda\ra  \Sym(\u_-)\otimes V_\lambda\] is given by
\[d_1: \u_-\otimes V_\lambda\ra \Sym_{\leq 1}(\u_-)\otimes  V_\lambda,\] which induces \[d_1^\vee: \Sym_{\leq 1}(\u_-^{\vee})\otimes V_\lambda^\vee\ra V_\lambda^\vee\otimes\u_-^{\vee}.\] This map can be explicitly described as follows. For any $k$-basis $y_1,\dots, y_n$ of $\u_-$ and dual basis $f_1,\dots, f_n$ of $\u_-^{\vee}$, for any $(c,e)\in \Sym_{\leq 1}(\u_-^\vee)=k\oplus \u_-^\vee$ and $v\in V_\lambda^\vee$, consider the map
\[d_1^\vee((c,e)\otimes v)= v\otimes e + \sum_{j=1}^n (y_jv)\otimes(cf_j). \]
Recall the functor $\E(\cdot): \Rep\,P\ra \Vect(X)$ in subsection \ref{subsection auto vb}.
By \cite{LP} Lemma 4.18, we have \[\E(\Sym_{\leq 1}(\u_-^{\vee}))\simeq \P^1_X,\] and by loc. cit. Proposition 4.27, under the functor $\E(\cdot)$ the differential $d_1^\vee$ is sent to the morphism
\[s^\ast-\tr{Id}^\ast: \P_X^1\otimes\V_\lambda^\vee \ra \V_\lambda^\vee\otimes\Omega_X^1 \]
(cf. \cite{LP} subsection 2.2 for the precise meaning of the notation $s^\ast-\tr{Id}^\ast$). Composed with the canonical morphism $ \V_\lambda^\vee\ra \P_X^1\otimes\V_\lambda^\vee$, this gives
the connection \[\nabla: \V_\lambda^\vee\ra \V_\lambda^\vee\otimes \Omega_X^1\] on the automorphic vector bundle $\V_\lambda^\vee$.
By our notation, we have
\[\nabla=(s^\ast-\tr{Id}^\ast)\circ pr_2^\ast=\varepsilon_1(pr_2^\ast)-pr_1^\ast,\] where $pr_i: P_X^1=\Spec\,\P_X^1\ra X$ are the natural projections. Recall $\P_X^1=\P_X/J_X^{[2]}$ and we can identify $\Omega_X^1=J_X/J_X^{[2]}$.
We have similar descriptions for all $d_a^\vee$ and $\E(d_a^\vee)$. \begin{proposition}[\cite{LP} Corollary 4.30 and Proposition 4.31]
Under the functor $\E(\cdot)$, we have 
\[DR(\V_\lambda^\vee,\nabla)\simeq\E(\tr{Std}^\bullet(V_\lambda^\vee)).\] Moreover, the functor $\E(\cdot)$ transfers  the Hodge filtration $C^\bullet$ on  $\tr{Std}^\bullet(V_\lambda^\vee)$ to the Hodge filtration on $DR(\V_\lambda^\vee,\nabla)$.
\end{proposition}
Similarly, we have
\begin{proposition}\label{prop p-curvature complex and p-std}
Under the functor $\E(\cdot)$, we have 
\[K(\V_\lambda^\vee,\psi)\simeq\E( p\tr{-Std}^\bullet(V_\lambda^\vee)), \] which is compatible on conjugate filtrations on both sides.
\end{proposition}
\begin{proof}
We need a similar description as above for the complex $p\tr{-Std}^\bullet(V_\lambda^\vee)$, in the way to take care of the $p$-curvature.
Replace $P$ by $P^{(p)}$ and consider the complex $p\tr{-Std}_\bullet(V_\lambda)$. All the differentials are of degree one. Consider the map $\psi_1: \Sym(\u_-^{(p)})\otimes \u_-^{(p)}\otimes V_\lambda\ra \Sym(\u_-^{(p)})\otimes V_\lambda$ which is given by $\psi_1: \u_-^{(p)}\otimes V_\lambda\ra \Sym_{\leq 1}(\u_-^{(p)})\otimes V_\lambda$. We can explicitly describe its dual 
\[\psi_1^\vee: \Sym_{\leq 1}(\u_-^{(p)\vee})\otimes V_\lambda^\vee\ra  V_\lambda^\vee\otimes\u_-^{(p)\vee} \]
similarly as $d_1^\vee$. 
We denote the composition of $\psi_1^\vee$ with the canonical map $V_\lambda^\vee\ra \Sym_{\leq 1}(\u_-^{(p)\vee})\otimes V_\lambda^\vee, v\mapsto 1\otimes v$ by $\psi^\vee$.
Then as above $\psi^\vee: V_\lambda^\vee\ra V_\lambda^\vee\otimes\u_-^{(p)\vee}$ can be explicitly described as \[\psi^\vee(v)=\sum_{i=1}^n(x_iv)\otimes f_i,\] for any $k$-basis $x_1,\dots, x_n$ of $\u_-^{(p)}$ and dual basis $f_1,\dots, f_n$ of $\u_-^{(p)\vee}$.

Under the functor $\E(\cdot)$, we have \[\E(\u_-^{(p)\vee})=Fr_X^\ast\Omega_X^1,\quad \tr{and}\quad \E(\Sym_{\leq 1}(\u_-^{(p)\vee}))=Fr_X^\ast \Sym_{\leq 1}\Omega_X^1=Fr_X^\ast \Gamma_{\leq 1}\Omega_X^1.\]
The map $\psi_1^\vee$ is sent to a morphism \[
\E(\psi_1^\vee): (Fr_X^\ast\Gamma_{\leq 1}\Omega_X^1) \otimes \V_\lambda^\vee\ra \V_\lambda^\vee\otimes Fr_X^\ast\Omega_X^1.\] 
Composed with the canonical map $\V_\lambda^\vee\ra (Fr_X^\ast\Gamma_{\leq 1}\Omega_X^1) \otimes \V_\lambda^\vee$,
we get a map \[\psi^0=\E(\psi^\vee): \V_\lambda^\vee\ra \V_\lambda^\vee\otimes Fr_X^\ast\Omega_X^1.\]
We claim that $\psi^0$ is the $p$-curvature $\psi_\nabla$ of $\nabla$. Indeed, 
recall by Proposition \ref{prop p-curvature} we have an
 isomorphism $Fr_X^\ast\Omega_X^1=F_{X/k}^\ast\Omega_{X'}^1\simeq J_X/(J_X^{[p+1]}+I\P_X)$. Using local coordinates, it is given by \[d\pi^\ast_X(x)\mapsto \tau^{[p]} \quad\tr{mod}\quad \,J_X^{[p+1]}+I\P_X,\] with $\tau=1\otimes x-x\otimes 1$, cf. \cite{OV} Proposition 1.6.
 Then using local coordinates, $\psi^0$ is given by \[\psi^0(v)=\sum_{i=1}^n\nabla(\partial_i)^p(v)\otimes \tau_i^{[p]},\]
 where $\partial_i=\frac{\partial}{\partial x_i}$ and $\tau_i=1\otimes x_i-x_i\otimes 1$. 
 Now the claim follows 
by the similar description of $\nabla$ (as above) in the proof of \cite{LP} Proposition 4.27, and the crystalline description of $\psi_\nabla$ in \cite{OV} Proposition 1.7.

Similar analysis holds for $\psi_a^\vee$. Thus we get $\E( p\tr{-Std}^\bullet(V_\lambda^\vee))\simeq K(\V_\lambda^\vee,\psi)$, which is compatible on conjugate filtrations on both sides.
\end{proof}

\subsection{Dual BGG complexes in characteristic $p$}\label{subsection dual BGG}
We transfer the results of subsection \ref{subsection BGG} to geometric setting.
\begin{theorem}\label{thm dual BGG}
	Assume that $\lambda$ is $p$-small.
We have a $C^\bullet$-filtered quasi-isomorphic embedding of complexes \[BGG(\mathcal{V}_\lambda^\vee)\hookrightarrow DR(\mathcal{V}_\lambda^\vee,\nabla),\] 
where $BGG(\mathcal{V}_\lambda^\vee)$ has the form
\[0\ra BGG^0(\mathcal{V}_\lambda^\vee)\ra BGG^1(\mathcal{V}_\lambda^\vee)\ra \cdots\ra BGG^n(\mathcal{V}_\lambda^\vee)\ra 0 \]
with $BGG^a(\mathcal{V}_\lambda^\vee)=\bigoplus_{w\in {}^JW(a)}\mathcal{W}^\vee_{w\cdot \lambda}$,
such that the associated graded complex \[\gr_{C}BGG(\mathcal{V}_\lambda^\vee)=\bigoplus_{w\in {}^JW}\mathcal{W}^\vee_{w\cdot \lambda}\] is a quasi-isomorphic direct summand of $\gr_{C}DR(\mathcal{V}_\lambda^\vee,\nabla)$ with trivial differentials. 
\end{theorem}
\begin{proof}
This follows from Theorem \ref{thm BGG} as in \cite{LP} Theorem 5.9. 
Let $H\in X_\ast(T)$ be the element defined by the conjugacy class of $\mu$. 
The $C^\bullet$-filtration on $BGG(\mathcal{V}_\lambda^\vee)$ is defined by
\[C^iBGG(\mathcal{V}_\lambda^\vee)=\bigoplus_{w\in {}^JW\atop w\cdot\lambda(H)\leq -i }\mathcal{W}^\vee_{w\cdot \lambda}.\]
 Note here we have the additional property that $BGG(\mathcal{V}_\lambda^\vee)\hookrightarrow DR(\mathcal{V}_\lambda^\vee,\nabla)$ is an embedding, as we work with PD differential operators. See also \cite{LP} Remark 5.20.
\end{proof}

We define a $D_\bullet$-filtration on $BGG(\mathcal{V}_\lambda^\vee)$  by 
\[D_iBGG(\mathcal{V}_\lambda^\vee)= \bigoplus_{w\in {}^JW\atop w\cdot\lambda(H)\geq -i}\mathcal{W}^\vee_{w\cdot \lambda},\]which induces a $D_\bullet$-filtration on $Fr_X^\ast BGG(\mathcal{V}_\lambda^\vee)$.
\begin{proposition}\label{prop dual p-BGG}
	Assume that $\lambda$ is $p$-small. We have a $D_\bullet$-filtered quasi-isomorphic embedding of complexes \[Fr_X^\ast\tr{BGG}(\mathcal{V}_\lambda^\vee)\hookrightarrow DR(\mathcal{V}_\lambda^\vee,\nabla),\] 
	such that $\gr_{D}Fr_X^\ast\tr{BGG}(\mathcal{V}_\lambda^\vee)$ is a quasi-isomorphic direct summand of $\gr_{D}DR(\mathcal{V}_\lambda^\vee,\nabla)$. 
\end{proposition}
\begin{proof}
This follows from Propositions \ref{prop p-BGG}, \ref{prop p-curvature complex and p-std} and Proposition \ref{prop Frob gauge complexes} (and its proof), as in the last theorem.
\end{proof}

\subsection{De Rham cohomology with coefficients}\label{subsection dR coh}
Now we can assemble previous results in subsections \ref{subsection std complexes}--\ref{subsection dual BGG} to prove the second main result of this section. In the following, for simplicity we write \[H_{\dR}^i(X, \mathcal{V}_\lambda^\vee)=H_{\dR}^i(X/k, (\mathcal{V}_\lambda^\vee,\nabla)).\] 

Suppose that $X$ is of PEL type. Assume  that $X$ is proper, and the reductive group $G$ defined by the PEL datum is connected. 
For a weight $\lambda\in X^\ast(T)_+$, let $|\lambda|=|\lambda|_L$ be the integer defined as in \cite{LS1} Definition 3.2. More precisely, if $G_{k}=(\prod_{\tau\in I}G_\tau)\times \G_m$ decomposes as a product local factors arising from the decomposition of the PEL datum, with induced dominant characters $\lambda_\tau$ which can be described as a tuple of integers $\lambda_\tau=(\lambda_{\tau,i_\tau})$ with $\lambda_{\tau,1}\geq \lambda_{\tau,2}\geq\cdots\geq \lambda_{\tau,r_\tau}$, then $\lambda=((\lambda_\tau)_\tau,\lambda_0)$,  \[|\lambda|=\sum_{\tau\in I}|\lambda_\tau|\quad \tr{with}\quad |\lambda_\tau|=\sum_{i_\tau}\lambda_{\tau,i_\tau}',\] where if $G_\tau\simeq \tr{Sp}_{2r_\tau}$ then $\lambda_{\tau,i_\tau}'=\lambda_{\tau,i_\tau}$ (which can be chosen such that $\lambda_{\tau,r_\tau}\geq 0$); if $G_\tau\simeq \GL_{r_\tau}$, let $\lambda_{\tau,r_{\tau}+1}$ be the unique even integer such that $1\geq \lambda_{\tau,r_{\tau}}-\lambda_{\tau,r_{\tau}+1}\geq 0$ and set $\lambda_{\tau,i_{\tau}}'=\lambda_{\tau,i_\tau}-\lambda_{\tau,r_{\tau}+1}$. For each $\tau\in I$ such that $G_\tau\simeq \tr{Sp}_{2r_\tau}$, assume that $\max(2,r_\tau)<p$.
\begin{theorem}\label{thm de Rham coh BGG}
	Keep the above assumptions and notations. Let $n=\dim\,X$.  Assume moreover $2n<p$.
	\begin{enumerate}
		\item For any $0\leq i\leq 2n$ and $\lambda\in X^\ast(T)_+$ such that $\lambda$ is $p$-small with $|\lambda|<p$, and $V_\lambda$ is defined over $\F_p$, then there is a natural $F$-zip structure on $H_{\dR}^i(X, \mathcal{V}_\lambda^\vee)$, which is induced by the cohomology of the de Rham $F$-gauge $(\mathcal{V}_\lambda^\vee, \nabla, C^\bullet, D_\bullet, \varphi)$.
		\item For the $\lambda$ as above, the $F$-zip structure on $H_{\dR}^i(X, \mathcal{V}_\lambda^\vee)$ is determined by the dual BGG complex $BGG(\mathcal{V}_\lambda^\vee)$ as follows. Let $H\in X_\ast(T)$ be the element defined by the conjugacy class of $\mu$. For each $a\in \Z$,  by the construction of dual BGG complexes, we have
		\[
		H^i\Big(X, \gr_C^aBGG(\mathcal{V}_\lambda^\vee)\Big) =\bigoplus_{w\in {}^JW \atop w\cdot\lambda(H)=-a}H^{i-\ell(w)}(X, \mathcal{W}^\vee_{w\cdot\lambda}). \]
		Then there is a commutative diagram of isomorphisms
		\[\xymatrix{
		Fr_k^\ast \gr_C^aH^i_{\dR}(X,\mathcal{V}_\lambda^\vee)\ar[r]^{\varphi_a}& \gr_D^aH^i_{\dR}(X,\mathcal{V}_\lambda^\vee)\\
			Fr_k^\ast H^i\Big(X, \gr_C^aBGG(\mathcal{V}_\lambda^\vee)\Big)\ar[u]_{\sim} \ar[r]^{\varphi_a}&
		 H^i\Big(X, \gr_D^aFr_X^\ast BGG(\mathcal{V}_\lambda^\vee)\Big)\ar[u]_{\sim}.
		}\] 
	\end{enumerate}
\end{theorem}
\begin{proof}
For (1), we use the theory of Kuga families to deduce that the Hodge-de Rham spectral sequence is degenerate. More precisely, under the assumption as above, there exists an integer $t_\lambda$ such that $\mathcal{V}_\lambda^\vee$ is isomorphic to an explicit direct summand of \[H^{n_\lambda}_\dR(\mathcal{A}^{n_\lambda}/X)(-t_\lambda)\] with compatible Hodge filtrations, where $n_\lambda=|\lambda|$ and $\mathcal{A}^{n_\lambda}\ra X$ is the $n_\lambda$-times fiber product of the universal abelian scheme $\mathcal{A}\ra X$, cf. \cite{LS1} Proposition 3.7. (For the non proper case, see \cite{Lan} and \cite{LS2} section 5.) Moreover, by \cite{LS1} Proposition 4.8, $H_{\dR}^i(X, \mathcal{V}_\lambda^\vee)$ is isomorphic to an explicit direct summand of \[H^{i+n_\lambda}_{\dR}(\mathcal{A}^{n_\lambda}/k)(-t_\lambda)\] for each $0\leq i\leq 2n$, and by \cite{LS1} Proposition 4.4 and Lemma 4.7,
 the Hodge-de Rham spectral sequence
\[E_1^{a,b}=H^{a+b}\Big(X,\gr_C^aDR(\V_\lambda^\vee,\nabla)\Big)\Rightarrow H^{a+b}_{\dR}(X,\V_\lambda^\vee)\] is isomorphic to a direct summand of the spectral sequence \[E_1^{a+n_\lambda,b}=H^{a+b+n_\lambda}\Big(\mathcal{A}^{n_\lambda},\Omega^{a+n_\lambda}_{\mathcal{A}^{n_\lambda}/k}\Big)(-t_\lambda)\Rightarrow H^{a+b+n_\lambda}_{\dR}(\mathcal{A}^{n_\lambda}/k)(-t_\lambda).\] Since the later degenerates at $E_1$ by Example \ref{example F-gauge}, so does the former.
Then we get the $F$-zip structure on cohomology by Proposition \ref{prop coh F-gauge}.

For (2),  by part (1), for each $a\in\Z$, we get an isomorphism
\[\varphi_a: Fr_k^\ast \gr_C^aH^i_{\dR}(X,\mathcal{V}_\lambda^\vee)\st{\sim}{\ra} \gr_D^aH^i_{\dR}(X,\mathcal{V}_\lambda^\vee).\] By 
Theorem \ref{thm dual BGG}, the embedding $\gr_CBGG(\mathcal{V}_\lambda^\vee)\hookrightarrow \gr_C DR(\V_\lambda^\vee,\nabla)$ is a quasi-isomorphic direct summand. Combined with the degeneracy of the Hodge-de Rham spectral sequence, we get a canonical isomorphism \[Fr_k^\ast H^i\Big(X, \gr_C^aBGG(\mathcal{V}_\lambda^\vee)\Big)\st{\sim}{\ra}
Fr_k^\ast \gr_C^aH^i_{\dR}(X,\mathcal{V}_\lambda^\vee).\] 
By
Proposition \ref{prop dual p-BGG},  the embedding $\gr_DFr_X^\ast BGG(\mathcal{V}_\lambda^\vee)\hookrightarrow \gr_D DR(\V_\lambda^\vee,\nabla)$ is a quasi-isomorphic direct summand. Combined with the degeneracy of the conjugate spectral sequence (which follows from that for the Hodge-de Rham spectral sequence by Proposition \ref{prop coh F-gauge}), we have also a canonical isomorphism
\[H^i\Big(X, \gr_D^aFr_X^\ast BGG(\mathcal{V}_\lambda^\vee)\Big)\st{\sim}{\ra}\gr_D^aH^i_{\dR}(X,\mathcal{V}_\lambda^\vee).\] Therefore the above $\varphi_a$ induces an isomorphism \[\varphi_a: Fr_k^\ast H^i\Big(X, \gr_C^aBGG(\mathcal{V}_\lambda^\vee)\Big)\st{\sim}{\ra}H^i\Big(X, \gr_D^aFr_X^\ast BGG(\mathcal{V}_\lambda^\vee)\Big)\] such that we have the above commutative diagram.
\end{proof}


\subsection{Mod $p$ crystalline comparison with small coefficients}
Finally, as in \cite{MT} section 6, we discuss briefly the mod $p$ crystalline comparison theorem with coefficients to indicate how to pass from mod $p$ de Rham cohomology to mod $p$ \'etale cohomology, following Faltings \cite{Fal, Fal99}. There are more recent developments by prismatic methods, for example see \cite{DLMS, GR}.   As currently a general comparison theorem for the $p$-torsion coefficients in $F\tr{-Gauge}_{dR}(X)$ is not available yet, we restrict to the Fontaine-Laffaille case, and moreover the absolutely unramified case to match our previous discussions. Thus we are in the setting of \cite{Fal}. Then the following discussions (with more details) already appeared in the works of Lan-Suh \cite{LS1, LS2}. 

Let $k$ be a perfect field of characteristic $p$, $W=W(k)$ and $E=W\otimes\Q$.
Consider $\X/W$ a proper smooth scheme over $W$ with $X/k$ its special fiber over $k$, and $\X_\eta$ its generic fiber over $E$.
Consider the category \[\MF^\nabla_{[0,p-2]}(\X_2/W_2)\] of $p$-torsion Fontaine modules associated to the $W_2=W_2(k)$-lift $\X_2$ of $X$. By Proposition \ref{prop MF dR gauge}, we can view it as a full subcategory of $F\tr{-Gauge}_{dR}(X)$ as we have a fully faithful functor \[ \MF^\nabla_{[0,p-2]}(\X_2/W_2)\hookrightarrow F\tr{-Gauge}_{dR}(X).\]  By \cite{Fal}, there is a fully faithful contravariant functor
\[\mathbb{D}: \MF^\nabla_{[0,p-2]}(\X_2/W_2)\lra \tr{Loc}(\X_\eta, \F_p), \]
where $\tr{Loc}(\X_\eta, \F_p)$ is the category of $\F_p$-local system on $\X_\eta$. Let $\mathbb{V}=\mathbb{D}^\ast$ be the functor which is the composition of $\mathbb{D}$ with the Pontryagin dual.
Moreover, if the $p$-torsion Fontaine module $\ul{\E}=(\E,\nabla, \Fil,\Phi) \in \MF^\nabla_{[0,p-2]}(\X_2/W_2)$ is associated with $\mathbb{L}\in \tr{Loc}(\X_\eta, \F_p)$, in the sense that \[\mathbb{V}(\ul{\E})=\mathbb{L},\] then for any integer $i\geq 0$ and $i+a\leq p-2$ where the integer $a$ is given by $\Fil^{a+1}=0$, we have an isomorphism (cf. \cite{Fal} Theorem 5.3)
\[\mathbb{V}(H^i_{\dR}(X, \E))\simeq H^i_{\et}(\X_{\eta, \ov{E}}, \mathbb{L}). \]

Now let $\X/W$ be the smooth integral canonical model of a PEL type Shimura variety, which we assume to be proper\footnote{This assumption can be removed, cf. \cite{LS2}, if we work with a smooth toroidal compactification and canonical extensions of automorphic vector bundles.} for simplicity. Suppose $2\dim \,X<p$.
Assume as before that $\lambda$ is $p$-small and $V_\lambda$ is defined over $\F_p$ (thus defined over $k$).
\begin{proposition}
If the weight $\lambda\in X^\ast(T)_+$ is $p$-small and the associated de Rham $F$-gauge
 $\ul{\V_\lambda^\vee}=(\V_\lambda^\vee, \nabla, C^\bullet, D_\bullet, \varphi)$ comes from an object $(\V_\lambda^\vee, \nabla, \Fil, \Phi)\in \MF^\nabla_{[0,p-2]}(\X_2/W_2)$, then we have 
\[ \mathbb{V}(\ul{\V_\lambda^\vee})=\mathbb{L}_\lambda^\vee,\]
where $\mathbb{L}_\lambda$ is the $\F_p$-local system on $\X_\eta$ associated to the $\F_p$-representation of $G$ with highest weight $\lambda$. In particular, for an integer $i\geq 0$ with $i+|\lambda|\leq p-2$, we have an isomorphism
\[\mathbb{V}(H^i_{\dR}(X, \V_\lambda^\vee))\simeq H^i_{\et}(\X_{\eta,\ov{E}}, \mathbb{L}_\lambda^\vee). \]
\end{proposition}
\begin{proof}
The first statement can be proved by the explicit realization of the automorphic vector bundle $\V_\lambda^\vee$ in terms of the Kuga families over $X$, and the second statement follows from the theorem of Faltings as in \cite{Fal}. See \cite{LS1} sections 3-5 for more details.
\end{proof}


\begin{thebibliography}{99}

\bibitem{And22}
Fabrizio Andreatta, \emph{On two period maps: Ekedahl-Oort and fine Deligne-Lusztig stratifications}, Math. Ann. 385, 511--550, 2023.

\bibitem{Bert} Pierre Berthelot, \emph{$\D$-modules arithm\'etiques. I. Op\'erateurs diff\'erentiels de niveau fini}, Ann. Sci. de l'\'E.N.S, tome 29, no. 2 (1996), 185-272.

\bibitem{BO} Pierre Berthelot, Arthur Ogus, \emph{Notes on crystalline cohomology}, (MN-21), Princeton University Press, 2015.

\bibitem{Bhatt} Bhargav Bhatt, \emph{Prismatic $F$-gauges}, lecture notes, available at \url{https://www.math.ias.edu./~bhatt/teaching.html}.

\bibitem{BL1} Bhargav Bhatt, Jacob Lurie, \emph{Absolute prismatic cohomology}, preprint, arXiv:2201.06120.

\bibitem{BL2}  Bhargav Bhatt, Jacob Lurie, \emph{The prismatization of $p$-adic formal schemes}, preprint, arXiv:2201.06124.


\bibitem{BMR} Roman Bezrukavnikov, Ivan Mirkovi\'c, Dmitriy Rumynin, \emph{Localization of modules for a semisimple Lie algebra in prime characteristic}, Ann. Math. 167 (2008), 945--991.


\bibitem{Dri18} Vladimir Drinfeld, \emph{A stacky approach to crystals}, in ``Dialogues Between Physics and Mathematics'' M.-L. Ge, Y.-H. He (eds.), 19-47, Springer, 2022.

\bibitem{Dri20} Vladimir Drinfeld, \emph{Prismatization}, Sel. Math. New Ser. 30, 49 (2024).

\bibitem{Dri23} Vladimir Drinfeld,  \emph{On Shimurian generalizations of the stack $\tr{BT}_1\otimes\F_p$}, arXiv:2304.11709.

\bibitem{Drin} Vladimir Drinfeld, \emph{Toward Shimurian analogs of Barsotti-Tate groups}, arXiv:2309.02346.

\bibitem{DLMS} Heng Du, Tong Liu, Yong Suk Moon, Koji Shimizu, \emph{Completed prismatic $F$-crystals and crystalline $\Z_p$-local systems}, Compos. Math. 160 (2024), 1101-1166.

\bibitem{Fal83} Gerd Faltings, \emph{On the cohomology of locally symmetric Hermitian spaces}, in Lecture Notes in Mathematics, vol. 1029,  55--98.

\bibitem{Fal} Gerd  Faltings, \emph{Crystalline cohomology and $p$-adic Galois-representations}, in ``Algebraic analysis, Geometry,
and Number theory'' (Baltimore, MD, 1988), 25-80, Johns Hopkins Univ. Press, Baltimore,  MD, 1989.

\bibitem{Fal99} Gerd  Faltings, \emph{Integral crystalline cohomology over very ramified valuation rings}, Journal of the AMS, vol. 12, no. 1, 117-144, 1999.

\bibitem{FC} Gerd Faltings, Ching-Li Chai, \emph{Degeneration of abelian varieties}, Erg. Math. Wiss. 3. folge, band 22, Springer-Verlag, 1990.

\bibitem{FJ}  Jean-Marc Fontaine, Uwe Jannsen, \emph{Frobenius gauges and a new theory of $p$-torsion
sheaves in characteristic $p$}, Documenta Math. 26 (2021), 65--101.

\bibitem{GM} Zachary Gardner, Keerthi Madapusi,  \emph{An algebraicity conjecture of Drinfeld and the moduli of $p$-divisible groups}, preprint, arXiv:2412.10226.

\bibitem{GoldringKoskivirta2019}
Wushi Goldring, Jean-Stefan Koskivirta,
\emph{Strata Hasse invariants, Hecke algebras and Galois
representations},
Invent. Math. 217(3):887--984, 2019.

\bibitem{GSQ} Michel Gros, Bernard Le Stum, Adolfo Quir\'os, \emph{A Simpson correspondence in positive characteristic}, Publ. RIMS Kyoto Univ. 46 (2010), 1-35.

\bibitem{GR} Haoyang Guo, Emanuel Reinecke, \emph{A prismatic approach to crystalline local systems}, Invent. Math. 236 (2024), 17-164.

\bibitem{IKY} Naoki Imai, Hiroki Kato, Alex Youcis, \emph{The prismatic realization functor for Shimura varieties of abelian type}, preprint, arXiv:2310.08472.

\bibitem{Ito}Kazuhiro Ito, \emph{Deformation theory for prismatic $G$-displays}, Forum of Mathematics, Sigma (2025), vol. 13:e61 1-57.

\bibitem{Jant} Jens Carsten Jantzen, \emph{Representations of Lie algebras in positive characteristic}, Adv. Stud. Pure Math. 40, Mathematical Society of Japan, Tokyo, 2004, 175--218.

\bibitem{Kat70} Nicholas M. Katz, \emph{Nilpotent connections and the monodromy theorem: applications of the result of Turrittin}, Publ. Math. IHES. 39 (1970), 175--232.

\bibitem{Kat72} Nicholas M. Katz, \emph{Algebraic Solutions of Differential Equations ($p$-curvature and the Hodge Filtration)},
Invent. Math. 18 (1972), 1--118.

\bibitem{Khan} Adeel A. Khan, \emph{Lectures on algebraic stacks}, available on the author's website.

\bibitem{KM}Wansu Kim, Keerthi Madapusi Pera, \emph{2-adic integral canonical models}, Forum Math. Sigma 4, e28, 2016.

\bibitem{Kis} Mark Kisin,  \emph{Integral models for Shimura varieties of abelian type}, J. Amer. Math. Soc. 23,  967-1012, 2010.

\bibitem{KP} Mark Kisin, Georgios Pappas,  \emph{Integral models of Shimura varieties with parahoric level structure}, Publ.Math. Inst. Hautes \'Etudes Sci. 128 (2018), no. 1, 121-218.
	
\bibitem{Lan} Kai-Wen Lan, \emph{Toroidal compactifications of PEL-type Kuga families}, Algebra Number Theory 6
(2012), no. 5, 885--966.

\bibitem{Lan13}Kai-Wen Lan, \emph{Arithmetic compcactifications of PEL-type Shimura varieties}, London Mathematical Society Monographs Series, vol. 36, Princeton University Press, Princeton, NJ, 2013.
	
\bibitem{LP} Kai-Wen Lan, Patrick Polo, \emph{Dual BGG complexes for automorphic bundles}, Math. Res. Lett. 25 (2018) 1, 85--141.

\bibitem{LS18}
Kai-Wen Lan, Beno\^{i}t Stroh, \emph{Compactifications of subschemes of integral models of Shimura varieties}, Forum Math. Sigma, 6:e18.

\bibitem{LS1}Kai-Wen Lan, Junecue Suh, \emph{Vanishing theorems for torsion automorphic sheaves on compact PEL-type Shimura varieties}, Duke Math. J. 161 (2012), no. 6, 1113-1170.

\bibitem{LS2}Kai-Wen Lan, Junecue Suh, \emph{Vanishing theorems for torsion automorphic sheaves on general PEL-type Shimura varieties}, Adv. Math. 242 (2013), 228-286.

\bibitem{LSZ1} Guitang Lan, Mao Sheng, Kang Zuo,  \emph{Non-abelian Hodge theory in positive characteristic via exponential
twisting}, Math. Res. Let. 22 (2015) 3, 859--879.

\bibitem{LSZ2}Guitang Lan, Mao Sheng, Kang Zuo, \emph{Semistable Higgs bundles, periodic Higgs bundles and representations
	of algebraic fundamental groups},  J. Eur. Math. Soc. 21 (2019), 3053-3112.

\bibitem{Lau} Eike Lau, \emph{Smoothness of the truncated display functor}, J. Amer. Math. Soc. 26 (2013), no. 1, 129-165.

\bibitem{Lov} Tom Lovering, \emph{Filtered $F$-crystals on Shimura varieties of abelian type}, preprint, arXiv:1702.06611.

\bibitem{MP}Keerthi Madapusi Pera,  \emph{Toroidal compactifications of integral models of Shimura varieties of Hodge type},  Ann. Sci. de l'\'E.N.S, vol. 52 (2019), 393-514.

\bibitem{Men} Max Menzies, \emph{The $p$-curvature conjecture for the non-abelian Gauss-Manin connection}, Harvard University Ph.D Thesis, arXiv:1912.05757.

\bibitem{MT} Abdellah Mokrane, Jacques Tilouine, \emph{Cohomology of Siegel varieties with $p$-adic integral coefficients
and applications}, in Ast\'erisque 280 ``Cohomology of Siegel varieties'', 1--95.
	
\bibitem{MW}
Ben Moonen, Torsten Wedhorn,	\emph{Discrete invariants of varieties in positive characteristic},  IMRN 72 (2004), 3855--3903.

\bibitem{Ogus} 
Arthur Ogus, \emph{Higgs cohomology, $p$-curvature, and the Cartier isomorphism}, Compo. Math. 140 (2004) 145--164.

\bibitem{OV} 
Arthur Ogus, Vadim Vologodsky, \emph{Nonabelian Hodge theory in characteristic $p$}, Publ. Math. IHES. 106 (2007), 1--138.

\bibitem{Oy}
 Hidetoshi Oyama, \emph{PD Higgs crystals and Higgs cohomology in characteristic $p$},  J. Algebraic Geom. 26 (2017), no. 4, 735--802.

\bibitem{PWZ11}
Richard Pink, Torsten Wedhorn, Paul Ziegler, \emph{Algebraic zip data}, Doc. Math. 16: 253--300, 2011.

\bibitem{PWZ15}
Richard Pink, Torsten Wedhorn, Paul Ziegler, \emph{$F$-zips with additional structure}, Pac. J. Math. 274(1):183--236, 2015.

\bibitem{PT}
Patrick Polo, Jacques Tilouine, \emph{Bernstein-Gelfand-Gelfand complex and cohomology of nilpotent  groups over $\Z_{(p)}$ for representations with $p$-small weights}, in Ast\'erisque 280 "Cohomology of Siegel varieties", 97--135.

\bibitem{Schep}
Daniel Schepler, \emph{Logarithmic nonabelian Hodge theory in characteristic $p$}, preprint, arXiv: 0802.1977.

\bibitem{SZ}
Xu Shen, Chao Zhang, \emph{Stratifications in good reductions of Shimura varieties of abelian type},  Asian J. Math. vol. 26, no. 2 (2022), 167-226.

\bibitem{Sta} The Stacks Project, \url{https://stacks.math.columbia.edu}

\bibitem{ViehmannWedhorn2013}
Eva Viehmann, Torsten Wedhorn,
\emph{Ekedahl-Oort and Newton strata for Shimura varieties of PEL type}, Math. Ann. 356 (2013) no.4, 1493--1550.

\bibitem{Wed}
Torsten Wedhorn, \emph{De Rham cohomology of varieties over fields of positive characteristic},
Nato Security through science series D-information and communication security, 16, 269--314, 2008.

\bibitem{WZ}
Torsten Wedhorn, Paul Ziegler,
\emph{Tautological rings of Shimura varieties and cycle classes of Ekedahl-Oort strata}, Algebra  Number Theory, vol. 17 (2023), no.4, 923--980.

\bibitem{Zhang2018EO}
Chao Zhang,
\emph{Ekedahl-Oort strata for good reductions of Shimura varieties of
Hodge type}, Canadian Journal of Mathematics, vol. 70, 451--480, 2018.

\end{thebibliography}
\end{document}